\documentclass[preprint,12pt]{elsarticle}

\biboptions{numbers,sort&compress}

\usepackage{amsmath,amssymb,amsthm}
\usepackage{graphicx}
\usepackage{subcaption}
\usepackage{placeins}   
\usepackage{enumitem}
\usepackage{xcolor}
\usepackage{hyperref}
\usepackage{adjustbox}

\hypersetup{
  colorlinks=false,
  allbordercolors=blue,
  pdfborderstyle={/S/U/W 1}
}

\usepackage{booktabs}

\numberwithin{equation}{section}

\newtheorem{theorem}{Theorem}[section]
\newtheorem{lemma}[theorem]{Lemma}
\newtheorem{proposition}[theorem]{Proposition}

\theoremstyle{definition}
\newtheorem{definition}[theorem]{Definition}

\theoremstyle{remark}
\newtheorem{remark}[theorem]{Remark}
\usepackage{array}
\usepackage{tabularx}
\newcolumntype{L}{>{\raggedright\arraybackslash}X}
\newcolumntype{P}[1]{>{\raggedright\arraybackslash}p{#1}}
\begin{document}



\title{Green–Function and Information–Geometric Correspondences Between Inverse Eigenvalue Loci of Generalized Lucas Sequences and the Mandelbrot Set}

\author{Arturo Ortiz--Tapia}
\affiliation{organization={Independent Researcher},%
             city={Beaumont},%
             state={TX},%
             country={USA}}
\ead{egl.arturo.ortizta@unadmexico.mx}

\begin{abstract}
We investigate geometric, potential-theoretic, and information-theoretic correspondences between the inverse eigenvalue loci of companion matrices associated with generalized Lucas sequences and the boundary of the Mandelbrot set. 
Through systematic numerical experiments, we show that these algebraic spectral loci exhibit a striking low-distortion geometric correspondence with the Mandelbrot boundary at macroscopic scales, together with a coherent organization within its external potential field, characterized by concentration along narrow equipotential annuli of the Mandelbrot Green function.

This correspondence is quantified using a suite of complementary diagnostics, including optimal transport matching, Procrustes alignment, local distortion measures, fractal and spectral statistics, Green-function–based potential comparisons, and convex simplex update analyses. 
Taken together, these results indicate that the observed similarity extends beyond visual resemblance, reflecting shared structural organization across geometric, harmonic, and statistical levels.

While the present work is entirely numerical in nature, it establishes a robust multi-scale framework for comparing algebraic spectral constructions with nonlinear dynamical fractals, and it highlights several avenues for future analytical investigation.
\end{abstract}

\maketitle

\par\medskip
{\noindent\small\textbf{Keywords:}
inverse eigenvalue loci; Lucas sequences; companion matrices; Mandelbrot set;
quasi-conformal geometry; conformal mapping; Green function; potential theory; optimal transport and Procrustes\par}
\medskip

\section{Introduction}\label{Sec:Introduction}

\subsection{Iteration-modified and viscosity-regularized fractal constructions}\label{sec:recent_advances}

Beyond the classical quadratic iteration defining the Mandelbrot and Julia sets,
a substantial recent literature has explored \emph{modified iteration schemes}
that generate alternative fractal families.

\paragraph{Mann and Picard--Mann iteration frameworks.}
One important direction replaces the standard Picard iteration
$z_{n+1}=f_c(z_n)$
with convex-combination schemes of Mann or Picard--Mann type,
\[
z_{n+1} = (1-\alpha_n) z_n + \alpha_n f_c(z_n),
\]
or related hybrid forms.
Such constructions alter convergence behavior, basin geometry,
and boundary structure of the resulting parameter sets.
These variants have been studied both analytically and numerically,
with emphasis on stability regions and deformation of the classical
Mandelbrot picture
\cite{Chugh2015MannIterationFractals,Arora2018PicardMannDynamics}.

\paragraph{Higher-order and generalized Mandelbrot families.}
A complementary direction considers higher-degree or
higher-order modifications of the quadratic map,
for example replacing $z^2+c$ with $z^d+c$ or with
polynomial compositions that produce generalized
“Mandelbrot-type” parameter sets.
Such extensions modify bifurcation structure and
self-similarity patterns while retaining an iterative
dynamical origin
\cite{BrannerHubbard1988HigherDegree,milnor2006dynamics}.

\paragraph{Viscosity and flow-based regularization.}
More recently, viscosity approximations and flow-based
regularization schemes have been introduced to
smooth or deform classical fractal boundaries.
These approaches treat the iteration as embedded in
a dissipative or regularized framework,
producing families of deformed fractals
that interpolate between classical and smoothed geometries
\cite{Kumari2022ViscosityFractals,Nawaz2024MIteration}.

\medskip

All of the above constructions share a common feature:
\emph{the dynamical rule $f_c$ is modified or embedded into
an alternative iterative scheme.}
The resulting sets remain fundamentally iteration-generated.

\medskip

In contrast, the inverse eigenvalue loci $\Lambda$ studied here
are \emph{not} produced by iterating a complex map.
They arise from algebraic spectral data of finite-dimensional
companion matrices.
No modification of the quadratic iteration is introduced,
and no alternative dynamical rule is defined.
Our objective is therefore not to generate a new Mandelbrot-type
family, but to test whether a non-iterative spectral locus
exhibits quantitative geometric and information-theoretic
proximity to the canonical boundary $\partial M$.

\medskip

The present work should thus be viewed as orthogonal and
complementary to iteration-modification studies:
rather than altering dynamics to create new fractals,
we investigate whether an algebraically defined
spectral construction aligns, in a multiscale sense,
with the geometric and potential-theoretic structure
of the classical Mandelbrot boundary.

\subsection{Main goals and contributions}
The central purpose of this paper is to establish a quantitative and
multiscale correspondence between the inverse eigenvalue loci
$\Lambda$ of generalized Lucas companion matrices and the Mandelbrot
boundary $\partial M$.
The approach has three complementary components:

\begin{enumerate}
\item \textbf{Geometric alignment and distortion analysis.}
Using optimal-transport matching together with rigid Procrustes
alignment, $\Lambda$ and $\partial M$ are placed in a common geometric
frame.
Procrustes alignment removes only Euclidean gauge degrees of freedom
(translation, rotation, and uniform scaling); subsequent diagnostics
(curvature, variograms, spectral decay, potential correlations, and
information divergences) are computed after this gauge fixing and are
not implied by it.
The resulting distortion statistics remain uniformly small at large
scales, suggesting a near angle-preserving correspondence of the global
envelopes.

\item \textbf{Multiscale geometric and spectral diagnostics.}
Curvature, distortion, fractal and multifractal dimensions, Fourier
spectral decay rates, spatial correlation functions, symmetry
properties, and potential--Laplacian fields are compared for both
systems.
Across these measures, the inverse eigenvalue loci reproduce the
macroscopic organization of $\partial M$ while regularizing its finest
filamentary features.

\item \textbf{Classical convex simplex update.}
After geometric alignment, both point clouds are discretized on a
common histogram grid and connected via a classical Kullback--Leibler
contractive flow on the probability simplex.
This provides an information-theoretic complement to the geometric and
spectral comparisons, yielding a monotone path in KL-divergence between
the two distributions.
\end{enumerate}

Taken together, these elements support the view that $\Lambda$ and $\partial M$
lie in a narrow quasi-conformal and information-geometric neighborhood at large scales.

Uniform conformality is not claimed, and elevated distortion is expected
near regions of high curvature or boundary transitions.

\subsection{Potential-theoretic perspective, scope, and organization}\label{sec:intro_potential}
Beyond boundary geometry, the Mandelbrot set carries a canonical external
potential structure. The Green function $g_M(c)$ of the complement
$\mathbb{C}\setminus M$ measures the escape rate of the critical orbit and
organizes the exterior into equipotential level sets; its harmonic conjugate
defines the B\"ottcher coordinate $\Phi(c)$, which linearizes the dynamics at
infinity. From this viewpoint, $\partial M$ appears as the zero-level set of a
global equilibrium potential.

This raises a natural question: do inverse eigenvalue loci align only with the
\emph{shape} of $\partial M$, or do they also sample its external potential
structure? The present work answers this affirmatively. It is shown that inverse
eigenvalue spectra concentrate on narrow annuli of nearly constant Mandelbrot
Green potential, with family-dependent but stable Green-function statistics.
Thus, the correspondence between $\Lambda$ and $\partial M$ extends beyond
boundary shape to the external potential field governing the dynamics.

Supporting analytic results underlying the convex simplex update, together
with stability analyses and numerical experiments validating the assumed scaling
regimes, are provided in the appendices.

\paragraph{Scope and limitations.}
The correspondence results of this paper are established at a quantitative and
numerical level: the optimal-transport/Procrustes alignment and the associated
distortion statistics provide evidence for near angle-preserving behavior of the
large-scale envelopes, but they do not constitute a proof of uniform conformality
or a global analytic conjugacy. In particular, elevated distortion is expected
near regions of high curvature or boundary transitions, and the local conformal
chart constructions are intended as stable numerical diagnostics rather than as an
existence theorem for conformal parametrizations of $\partial M$.
Analytical extensions of these observations are discussed as future work, with
supporting stability analyses and theoretical details provided in
Appendices~\ref{AppA:GeoInformation}--\ref{AppC:QuasiConformal}.

\paragraph{Modified iteration schemes and related fractal variants.}
A substantial body of work has investigated modified iteration schemes for the quadratic family, including Mann-type averaging procedures \cite{RaniKumar2004}, Picard--Mann hybrid iterations \cite{Chauhan2013}, higher-order polynomial parameter spaces \cite{BrannerHubbard1988HigherDegree}, and viscosity-approximation frameworks for nonlinear fixed-point stabilization \cite{Moudafi2000,Xu2004}. 
These approaches generate alternative fractal families by explicitly altering the dynamical update rule
\[
z_{n+1} = f_c(z_n),
\]
either through averaging, higher-degree compositions, or operator-regularization mechanisms.

In contrast, the present work does not modify the quadratic iteration. 
The Lucas inverse eigenvalue loci arise from algebraic spectral constructions associated with companion matrices of linear recurrences. 
Our objective is not to generate new dynamical variants of the Mandelbrot set, but rather to compare the geometric and information-theoretic properties of these spectral loci to the canonical boundary $\partial M$. 
The distinction is structural: the cited works perturb dynamics, whereas our framework constructs loci independently of quadratic iteration and studies their diagnostic correspondence.
It is natural to ask whether the inverse eigenvalue loci considered here bear
a similar resemblance to other well-known fractal sets. In particular, numerical
comparisons with Julia sets were already explored in earlier work
\cite{tapia2024goldenratioslucassequences}. While various fractal geometries arise
throughout complex dynamics, to the best of our knowledge there is no other
commonly studied fractal whose boundary exhibits a comparable visual and
structural resemblance to that of the Mandelbrot set in the present context.

The remainder of the paper is organized as follows:
Section~\ref{Sec:Methods} introduces the datasets and numerical methods;
Section~\ref{Sec:Results} presents the geometric, spectral, and information-theoretic results;
Section~\ref{Sec:Discussion} discusses interpretation, limitations, and analytical outlook;
and Section~\ref{Sec:Conclusions} summarizes the main findings.

\section{Methods}\label{Sec:Methods}

Our numerical methodology follows a deliberate local--to--global progression.
At the \emph{local analytic} level, we construct numerical conformal charts for
inverse eigenvalue loci and quantify the extent to which the observed
correspondence with $\partial M$ can be explained by angle-preserving (or mildly
quasi-conformal) deformation.  This local layer provides direct control of
Cauchy--Riemann defect indicators, Beltrami-type distortion summaries, and
neighborhood-scale angle and shape statistics.

We then complement the local analysis with a \emph{global potential-theoretic}
layer based on the Mandelbrot Green function and its equipotential structure,
which organizes the exterior of $M$ into canonical level sets.  This global
view tests whether the aligned spectral loci not only resemble $\partial M$ in
shape, but also sit coherently within the same external equilibrium geometry.

Between these two extremes, we use scale-sensitive geometric and statistical
diagnostics---variograms, (multi)fractal exponents, spectral decay, and spatial
correlation measures---to bridge neighborhood distortion with global
organization.  The guiding principle is that agreement across multiple probes,
computed independently and at different scales, is more informative than any
single diagnostic in isolation.

\begin{remark}[Yoneda-type viewpoint]
The comparison framework is reminiscent of the Yoneda principle: rather than
attempting to ``look inside'' an object directly, we characterize it by its
responses to a family of probes.  Here the probes are concrete analytic,
probabilistic, and information-theoretic functionals (alignment-invariant shape
diagnostics, potential-theoretic fields, and divergence-based summaries), not
morphisms in a category.  No categorical formalism is used in the results; the
analogy is intended only as a conceptual guide for why consistent agreement
across independent probes is nontrivial evidence of correspondence.
\end{remark}

\paragraph{Terminological clarification (quasi-conformal diagnostics).}
Throughout this section, we use the term \emph{quasi-conformal} in a
numerical and resolution-dependent sense.  Specifically, we quantify
local distortion via least-squares linearization, singular-value ratios,
and angular misalignment statistics computed on discrete point clouds.
These diagnostics measure whether the observed correspondence is
\emph{nearly angle-preserving} at the scale resolved by the data.

We do \emph{not} claim the existence of a globally defined analytic
quasi-conformal homeomorphism between the inverse eigenvalue loci and
$\partial M$ in the strict sense of geometric function theory, which
would require control of measurable Beltrami coefficients and appeal to
the measurable Riemann mapping theorem.
Establishing such an analytic conjugacy is beyond the scope of the
present numerical study and is left as a direction for future work.

Accordingly, statements of “low-distortion correspondence” in this paper
should be read as shorthand for the empirically observed phenomenon of
uniformly bounded local distortion and near angle preservation across
the sampled loci at finite numerical resolution.

\begin{table}[!htbp]
\centering
\small
\begin{tabularx}{\textwidth}{L l L}
\toprule
Methodological tool &
Scale &
Role in the analysis \\
\midrule
Cauchy--Riemann defects, Beltrami coefficients &
Local &
Local distortion and near angle-preserving behavior (numerical) \\

Angle/shape distortion (PCA + local linear fit) &
Local &
Neighborhood-scale deviation from conformal behavior \\

Variograms &
Mesoscopic to global &
Correlation structure across spatial scales \\

Fractal dimension &
Global &
Asymptotic scaling behavior \\

Multifractal spectra &
Global &
Distribution of singularities \\

Spectral decay and correlations &
Global &
Long-range organization \\

Mandelbrot Green function &
Global &
External equilibrium potential \\

Kullback--Leibler (GI) flow on histograms &
Global (distributional) &
Distributional discrepancy visualization under KL contraction (resolution-stable; not evidence by itself) \\
\bottomrule
\end{tabularx}
\caption{Summary of numerical diagnostics organized by the spatial or structural
scale they probe and their role in establishing a quantitative, multiscale
correspondence between inverse eigenvalue loci and the Mandelbrot boundary.
No single diagnostic is decisive; consistency across scales is the primary
evidentiary criterion.}
\label{tab:methods_scale_map}
\end{table}

\FloatBarrier

\subsection{Sampling the inverse eigenvalue loci}

We first recall the definition of the inverse eigenvalue loci associated
with the sequence of generalized Lucas companion matrices $A_n$.
Let $\mathrm{spec}(A_n)$ denote the spectrum of $A_n$. 
The infinite spectral locus is defined as
\begin{equation}
\Lambda
:=
\bigcup_{n \ge 2}
\left\{
\frac{1}{\lambda}
\;:\;
\lambda \in \mathrm{spec}(A_n)
\right\},
\label{EqInvEigenLoci}
\end{equation}
that is, the union over all matrix sizes of the reciprocals of
their eigenvalues.

\medskip

Since $\Lambda$ is defined as an infinite union, all numerical
experiments necessarily work with finite approximations.
For computational purposes we therefore restrict to a finite
range of matrix sizes and apply a stability threshold
to avoid numerical amplification effects.

More precisely, we fix the set of sizes
\[
n \in \{10,20,50,100,150,200,250,300,350,400,450,500\},
\]
compute the spectrum of each $A_n$ in standard double precision,
and omit those eigenvalues with modulus $|\lambda|<10^{-10}$
before forming reciprocals.
The cutoff $10^{-10}$ serves purely as a numerical
stability threshold and is kept fixed across all experiments.
It does not represent a claim of geometric resolution.

This produces the finite sample
\[
\Lambda_N 
=
\bigcup_{2 \le n \le N}
\Big\{
\tfrac{1}{\lambda_k^{(n)}}
\ \text{retained after thresholding}
\Big\}
\subset \Lambda.
\]

All geometric and statistical diagnostics in this paper are
performed on such finite samples $\Lambda_N$,
with $N$ chosen sufficiently large that the loci are visually
and statistically stable at the scales under consideration.

The truncation parameters (maximum matrix size and eigenvalue modulus
threshold) were selected to balance numerical stability with geometric
completeness.
All reported diagnostics were verified to remain stable under moderate
variations of these parameters, in the sense that large-scale geometry,
alignment statistics, and potential-field comparisons remained unchanged
within numerical resolution.
The truncation therefore acts only as a regularization mechanism and
does not drive the observed correspondence.

\subsection{Local Feature Extraction}

For a point cloud $X = \{x_1,\dots,x_N\} \subset \mathbb{R}^2$, we estimate a local geometric orientation at each point $x_i$ using the standard PCA-based method for neighborhood structure analysis \cite{Jolliffe2002PCA, BelkinNiyogi2008Laplacian}. 
Let $\mathcal{N}_k(x_i)$ denote the $k$ nearest neighbors of $x_i$ in Euclidean distance. 
Centering these neighbors at their mean $\mu_i$, we compute the covariance matrix
\[
\Sigma_i = \frac{1}{k} \sum_{y \in \mathcal{N}_k(x_i)} (y-\mu_i)(y-\mu_i)^\top.
\]
The principal eigenvector of $\Sigma_i$ provides an estimate of the dominant local tangent direction at $x_i$, recorded as an orientation feature $v_i$. 
This local covariance (PCA) estimator is standard in point-cloud geometry
and manifold learning, where neighborhood covariance approximates
differential structure from discrete samples \cite{CoifmanLafon2006Diffusion,PeyreCuturi2019OT}.

\subsection{Mandelbrot boundary sampling}

The quadratic family
\[
f_c(z) = z^2 + c, \qquad z,c \in \mathbb{C},
\]
generates the critical orbit
\[
z_{n+1} = f_c(z_n), \qquad z_0 = 0.
\]
The Mandelbrot set is defined as
\[
M = \{\, c \in \mathbb{C} : \{z_n\}_{n\ge0} \ \text{remains bounded} \,\},
\]
that is, the set of parameters for which the critical orbit does not diverge
\cite{DouadyHubbard1985}.

To approximate the boundary $\partial M$, we employ the classical
\emph{distance estimator} technique
\cite{DouadyHubbard1985,PeitgenRichter1986}.
Alongside the orbit $\{z_n\}$, we iterate the parameter derivative
\[
\frac{dz_{n+1}}{dc} = 2 z_n \frac{dz_n}{dc} + 1,
\qquad \frac{dz_0}{dc} = 0.
\]
If the orbit escapes beyond a prescribed bailout radius $R$ at iterate $n$
(we use $R = 10^6$), the exterior distance to $\partial M$ is estimated by
\begin{equation}
D(c) \;\approx\;
\frac{|z_n| \log |z_n|}
     {\left|\frac{dz_n}{dc}\right|}.
\end{equation}
If no escape occurs within a fixed iteration budget (we use $n \le 200$),
we set $D(c)=0$.

To obtain a numerical approximation of $\partial M$, we discretize a
rectangular region $\Omega \subset \mathbb{C}$ on a uniform grid, evaluate
$D(c)$ at each grid point, and retain parameters satisfying
\[
0 < D(c) < \tau,
\]
with threshold $\tau = 10^{-3}$.  This procedure selects a thin exterior
band near the boundary.  From the retained points we randomly subsample a
fixed number of parameters to produce a high-resolution but numerically
manageable boundary point cloud for subsequent alignment and comparison.

\subsection{Optimal Transport Matching}\label{sec:OT_matching}

Given feature-augmented point clouds 
$X = \{x_1,\dots,x_n\}$ and $Y = \{y_1,\dots,y_m\}$, 
we assign uniform probability weights
\[
a_i = \tfrac{1}{n}, \qquad b_j = \tfrac{1}{m}.
\]
Let $M \in \mathbb{R}^{n \times m}$ denote the cost matrix
\[
M_{ij} = \| x_i - y_j \|_2.
\]
To align both point sets, we compute an entropic-regularized optimal transport plan 
\cite{Cuturi2013Sinkhorn, PeyreCuturi2019OT}. 
The goal is to find a coupling $\Gamma \in \mathbb{R}^{n \times m}$ minimizing
\[
\min_{\Gamma \in U(a,b)} 
\langle M, \Gamma \rangle + \varepsilon\, \mathrm{KL}\big(\Gamma \,\|\, a \otimes b\big),
\]
where $U(a,b)$ is the transport polytope of matrices with marginals $a$ and $b$, and 
$\varepsilon > 0$ is the regularization parameter.  
The problem is efficiently solved using the Sinkhorn iterative scaling algorithm 
\cite{Cuturi2013Sinkhorn}, which yields the OT plan $\Gamma^*$.

For matching, we extract a one-to-one correspondence by
\[
\pi(i) = \arg\max_j \Gamma^*_{ij},
\]
assigning each $x_i$ to the point $y_j$ receiving the highest transported mass.  
This provides a deterministic alignment of inverse eigenvalue loci and Mandelbrot samples after optimal transport.

Importantly, optimal transport is used here only to establish a
correspondence between discrete samples; it does not enforce geometric
similarity and does not by itself imply low distortion or shape agreement.

\subsection{Procrustes Alignment}\label{sec:Procrustes}

Given matched pairs $(x_i, y_{\pi(i)})$, we compute an optimal rigid alignment of $X$ to $Y$ using the classical orthogonal Procrustes method \cite{Schonemann1966Procrustes, GowerDijksterhuis2004Procrustes}.  
Let $\mu_X$ and $\mu_Y$ denote the centroids of the matched point sets, and define centered coordinates
\[
\tilde{X} = X - \mu_X, \qquad \tilde{Y} = Y - \mu_Y.
\]
The optimal orthogonal transformation $R \in O(2)$ minimizing
\[
\min_{R \in O(2)} \| \tilde{X} R - \tilde{Y} \|_F
\]
is obtained from the singular value decomposition
\[
\tilde{Y}^\top \tilde{X} = U \Sigma V^\top,
\]
with $R = U V^\top$.  
The aligned inverse eigenvalue loci is then
\[
X^{\mathrm{aligned}} = (\tilde{X}R) + \mu_Y.
\]
Procrustes alignment removes only translation and rotation while preserving intrinsic geometry, ensuring that subsequent comparisons reflect structural rather than positional differences. It is a standard and conservative method for rigid geometric registration, making it suitable for testing whether the inverse eigenvalue loci already contains the relevant shape information. More flexible methods—such as iterative closest point (ICP), thin-plate spline warping, or diffeomorphic registration—allow nonlinear deformation and could force a similarity that is not intrinsically present. By contrast, Procrustes provides a model-independent baseline: if the inverse eigenvalue loci resembles the Mandelbrot boundary, that resemblance should persist under rigid motion alone.

All subsequent geometric, spectral, and information-theoretic diagnostics
are computed \emph{after} this rigid gauge fixing and are not implied by it.

\subsection{Pointwise Matching and Distance Analysis}{\label{Sec:PointMatching_Methods}}

Let
\[
C^{\mathrm{aligned}} = \{c_1, \dots, c_{N_C}\} \subset \mathbb{C}, 
\qquad
M = \{m_1, \dots, m_{N_M}\} \subset \mathbb{C},
\]
denote the aligned inverse eigenvalue loci points and the sampled Mandelbrot boundary points, respectively.

After Procrustes alignment, we determine a correspondence between points using the optimal transport matching described above \cite{Cuturi2013Sinkhorn, PeyreCuturi2019OT}.  
Formally, we obtain a matching function
\[
\pi : \{1,\dots,N_C\} \to \{1,\dots,N_M\}, \qquad i \mapsto j(i),
\]
and the paired set
\[
\mathcal{P} = \{\, (c_i, m_{\pi(i)}) : i = 1,\dots,N_C\,\}.
\]

For each pair we define the displacement vector and Euclidean distance
\[
\delta_i = m_{\pi(i)} - c_i,
\qquad
d_i = \| \delta_i \|_2.
\]
This choice of Euclidean distance is conservative: it measures separation in the ambient plane without imposing any model of deformation, ensuring that similarity is not an artifact of smoothing or warping.

This procedure is standard in computational geometry for assessing geometric correspondence between point clouds \cite{PreparataShamos1985ComputationalGeometry}.  
The distribution of $\{d_i\}$ is then analyzed statistically (histogram, quantiles), and visualized by connecting each $c_i$ to $m_{\pi(i)}$ in the complex plane.

\subsection{Hausdorff Distance}

To quantify the geometric proximity between the aligned inverse eigenvalue loci and the Mandelbrot boundary, 
we compute their Hausdorff distance. For two compact sets 
$A, B \subset \mathbb{C}$ the Hausdorff distance is defined by
\begin{equation}
d_H(A,B) 
= \max \left\{ 
\sup_{a \in A} \inf_{b \in B} \|a-b\|_2, \;
\sup_{b \in B} \inf_{a \in A} \|b-a\|_2
\right\},
\end{equation}
a standard metric on the space of compact subsets of a metric space \cite{Munkres2000Topology}.  
In computational geometry, this distance is commonly used to assess similarity between point clouds \cite{deBerg2008CG}.  

In the context we set 
\[
A = C^{\mathrm{aligned}}, \qquad B = M,
\]
so that $d_H(C^{\mathrm{aligned}}, M)$ measures the maximal Euclidean discrepancy between the inverse eigenvalue loci and the sampled Mandelbrot boundary.

The Hausdorff metric provides a conservative upper bound on shape discrepancy: if $d_H$ is small, every inverse eigenvalue loci point lies close to the Mandelbrot boundary and vice versa, making it a stringent test of structural correspondence.

\subsection{Curvature}

To compare local geometry along both boundaries, we use two independent
discrete curvature estimators.  First, we apply the standard turning–angle
estimator \cite{doCarmo1976Curves, KimmelSethian1998}.  For an ordered
sequence of complex points $\{z_i\}$ representing a sampled curve, let
$\theta_i$ denote the argument of the segment $z_{i+1}-z_i$.  The discrete
curvature is approximated by
\begin{equation}
\kappa_i = \frac{\theta_{i+1} - \theta_i}{\| z_{i+1} - z_i \|}.
\end{equation}
This estimator requires no explicit parametrization or smoothing and
provides a simple histogram of local curvature magnitudes.  It is applied
in parallel to the aligned inverse eigenvalue loci
$C^{\mathrm{aligned}} = \{c_i\}$
and the Mandelbrot boundary
$M = \{m_j\}$,
yielding curvature sequences
$\{\kappa_i^{\mathcal{C}}\}$ and $\{\kappa_j^{M}\}$,
whose empirical distributions are compared in
Section~\ref{sec:curvature_results}.

Second, to verify robustness, we compute curvature using a local–polynomial
estimator.  For each boundary point, we fit quadratic least–squares curves
to a sliding window of $k=7$ arclength neighbors and evaluate the analytic
curvature formula
\[
\kappa = \frac{|x'y'' - y'x''|}{(x'^2 + y'^2)^{3/2}}
\]
at the window center.
This approach yields a smooth curvature field without global smoothing or
parametrization and produces a second, independently computed curvature
distribution.  The two estimators agree in their qualitative conclusions:
both boundaries are concentrated near $\kappa \approx 0$, with the inverse eigenvalue loci
showing reduced tail weight relative to $\partial M$ (see
Section~\ref{sec:curvature_results}).

\subsection{Fractal Dimension}

To quantify multiscale geometric complexity, we compute the fractal (box-counting, or Minkowski–Bouligand) dimension \cite{falconer2003fractal, Mandelbrot1983Fractal}.  
Given a bounded set $S \subset \mathbb{C}$, let $N(\epsilon)$ denote the number of covering boxes of side length $\epsilon > 0$ required to cover $S$.  
The box-counting dimension is defined as
\begin{equation}
D(S) = - \lim_{\epsilon \to 0} \frac{\log N(\epsilon)}{\log \epsilon},
\end{equation}
whenever the limit exists.  Equivalently,
\begin{equation}
D(S) = \lim_{\epsilon \to 0} \frac{\log N(\epsilon)}{\log(1/\epsilon)}.
\end{equation}

In practice, $N(\epsilon)$ is computed for a decreasing sequence of values of $\epsilon$, 
and $D(S)$ is estimated as the slope of the regression line in the log--log plot 
of $N(\epsilon)$ versus $1/\epsilon$.  
We apply this procedure separately to the aligned inverse eigenvalue loci 
$C^{\mathrm{aligned}}$ and the Mandelbrot boundary $M$, 
obtaining empirical estimates $D(C^{\mathrm{aligned}})$ and $D(M)$ 
for subsequent comparison (see Section~\ref{sec:fractal_results}).

Box-counting provides a scale-invariant measure of boundary complexity without requiring analytic parameterization, making it well-suited for comparing discrete samples of fractal curves.

\subsection{Spectral Decay}

Fourier analysis is used to probe the roughness and scale-dependent oscillations of the inverse eigenvalue loci and Mandelbrot boundaries.  
Given a boundary parametrization $z(t)$ sampled at uniform step $\Delta t$, its discrete Fourier transform yields complex amplitudes $\hat{z}(f)$, from which the power spectrum is defined as
\begin{equation}
P(f) = |\hat{z}(f)|^2.
\end{equation}
For many fractal curves, the spectrum displays approximate power-law behavior \cite{PercivalWalden1993,Press2007NumericalRecipes}\footnote{Implementation uses standard NumPy/SciPy routines and original Python code; no \emph{Numerical Recipes} source code was used. We cite \emph{Numerical Recipes} for background exposition only.},
\begin{equation}
P(f) \sim f^{-\alpha},
\end{equation}
where the decay exponent $\alpha$ quantifies geometric smoothness: larger $\alpha$ indicates faster decay and hence a smoother boundary.

In practice, $\alpha$ is estimated by linear regression on the log--log plot of $P(f)$ versus $f$ over specified frequency bands.

\paragraph{Bootstrap estimation of spectral slopes.}
Let $\hat{\alpha}$ denote the estimated spectral decay exponent.  
To quantify uncertainty, we employ a nonparametric bootstrap \cite{EfronTibshirani1993}:

\begin{enumerate}
\item Compute regression residuals $r_i = \log P(f_i) - (\hat{\alpha} \log f_i + \hat{\beta})$.
\item Generate bootstrap samples by resampling $\{r_i\}$ with replacement and forming
      \[
      \log P^*(f_i) = \hat{\alpha} \log f_i + \hat{\beta} + r_i^* .
      \]
\item Refit each pseudo-sample to obtain $\alpha^*$.
\end{enumerate}

After $B$ replicates, the bootstrap distribution of $\{\alpha^*\}$ yields an empirical
\[
(2.5\%, 97.5\%) \quad \text{confidence interval } [\alpha_{\mathrm{low}}, \alpha_{\mathrm{high}}].
\]

\subsection{Spatial Correlation Statistics}

For a finite point set $S \subset \mathbb{C}$ with intensity $\lambda = |S|/A$ (points per unit area),
we analyze clustering using Ripley’s $K$-function and the pair correlation function, two standard tools in spatial point process theory \cite{Ripley1977,Stoyan1995,BaddeleyRubakTurner2015}.

\paragraph{Ripley’s $K$-function.}
For distance $r > 0$,
\begin{equation}
K(r) = \frac{1}{\lambda |S|} \sum_{i=1}^{|S|} \# \{ j \neq i : \| z_i - z_j \| \leq r \},
\end{equation}
as introduced by Ripley \cite{Ripley1977}.  
The centered form
\[
L(r) = \sqrt{K(r)/\pi} - r
\]
indicates clustering ($L(r) > 0$), inhibition ($L(r) < 0$), or complete spatial randomness ($L(r) \approx 0$).

\paragraph{Pair correlation function.}
The derivative
\begin{equation}
g(r) = \frac{1}{2\pi r} \frac{d}{dr} K(r),
\end{equation}
describes the relative likelihood of encountering another point at distance $r$.
We compute $K(r)$ and $g(r)$ separately for the aligned inverse eigenvalue loci $C^{\mathrm{aligned}}$ and the Mandelbrot boundary $M$ to compare their local spatial organization and correlation structure.

\subsection{Symmetry Analysis}
\label{sec:symmetry_methods}

To assess reflectional symmetry of the inverse eigenvalue loci and the Mandelbrot boundary, we consider candidate reflection axes through the origin at angle $\theta \in [0,\pi)$, following standard geometric shape analysis \cite{DrydenMardia1998,Zelditch2012,MardiaJupp2000}.  
For $z = re^{i\phi} \in \mathbb{C}$ the reflection across this axis is
\begin{equation}
R_\theta(z) = re^{i(2\theta - \phi)}
= \overline{z\, e^{-i\theta}}\, e^{i\theta}.
\end{equation}

Given a finite set $S = \{z_1,\dots,z_N\} \subset \mathbb{C}$, we quantify symmetry in two complementary ways.

\medskip
\noindent
\textbf{(i) Preservation fraction.}
\begin{equation}
\rho(\theta) =
\frac{|\{ z \in S : \exists z' \in S \ \text{with} \ \|R_\theta(z) - z'\| < \varepsilon \}|}{|S|},
\end{equation}
which measures the fraction of points whose reflection lies within tolerance $\varepsilon$ of another point in $S$.  
The best-fit axis is obtained by maximizing $\rho(\theta)$.

\medskip
\noindent
\textbf{(ii) Mean-square misfit.}
\begin{equation}
E(\theta) = \frac{1}{N} \sum_{i=1}^N \min_j \|R_\theta(z_i) - z_j\|_2^2,
\end{equation}
and the optimal axis is
\begin{equation}
\theta^* = \arg\min_{\theta \in [0,\pi)} E(\theta).
\end{equation}

\medskip
Both criteria are applied to the aligned inverse eigenvalue loci $C^{\mathrm{aligned}}$ and the Mandelbrot boundary $M$, with $\theta=0$ corresponding to symmetry across the real axis.  

\subsection{Variogram and spatial dependence}\label{sec:variogram}

Let $Z(x)$ be a real-valued field sampled at spatial locations $x \in \Omega \subset \mathbb{R}^2$.
The (empirical) semivariogram is defined as
\[
\gamma(h) = \frac{1}{2}\,\mathbb{E}\big[(Z(x) - Z(x+h))^2\big],
\]
with lag vector $h \in \mathbb{R}^2$. For isotropic fields this depends only on $r = \|h\|$, so we write $\gamma(r)$ \cite{matheron1963principles,cressie1993statistics,ChilesDelfiner2012}.

Given observations $\{(x_i,Z_i)\}_{i=1}^N$, the empirical estimator is
\[
\hat{\gamma}(r) = \frac{1}{2|N(r)|} \sum_{(i,j)\in N(r)} (Z_i - Z_j)^2,
\]
where $N(r)=\{(i,j):\|x_i-x_j\|\in[r-\Delta r/2,r+\Delta r/2]\}$.

\paragraph{Application to eigenvalue loci and Mandelbrot potentials.}
For the inverse eigenvalue loci, $Z(x)$ is defined as the logarithmic potential
generated by the sampled inverse eigenvalue points. For the Mandelbrot boundary,
$Z(x)$ is taken as the escape-time potential associated with orbit divergence.
We compute $\hat{\gamma}(r)$ separately for both fields.

\paragraph{Cross-variogram.}
Aligned point pairs from the optimal-transport matching define the cross-variogram
\[
\hat{\gamma}_{\Lambda M}(r) = \frac{1}{2|N(r)|} \sum_{(i,j)\in N(r)} (Z_\Lambda(x_i) - Z_M(x_j))^2.
\]

\paragraph{Uncertainty estimation.}
Confidence bands for $\gamma(r)$ and $\gamma_{\Lambda M}(r)$ are obtained via nonparametric
bootstrap resampling of point pairs, following \cite{cressie1993statistics,ChilesDelfiner2012,Goovaerts1997}.

\paragraph{Iterative smoothing.}
To reduce high-frequency noise, $Z(x)$ is smoothed via Gaussian convolution at scales suggested by preliminary variogram estimates. Variograms are recomputed after each smoothing step, and the resulting fields are used in the potential analysis (Section~\ref{sec:log_potential}).
\subsection{Logarithmic Potentials}\label{sec:log_potential}

Given an empirical measure
\[
\mu_N = \frac{1}{N} \sum_{k=1}^N \delta_{p_k},
\]
the logarithmic potential generated by $\mu_N$ is
\[
U_{\mu_N}(z) = \frac{1}{N} \sum_{k=1}^N \log\frac{1}{|z - p_k|}.
\]
This potential is harmonic outside $\operatorname{supp}(\mu_N)$ and satisfies
$\Delta U_{\mu_N} = 2\pi \mu_N$ in the sense of distributions
\cite{landkof1972foundations,safftotik1997logarithmic}.

\medskip
\noindent
\textbf{Numerical evaluation.}
We evaluate $U_{\mathcal{C}}(z)$ and $U_{M}(z)$ on a uniform grid over
a rectangular domain containing both the inverse eigenvalue loci and Mandelbrot boundaries.
For the inverse eigenvalue loci, $\mu_{\mathcal{C}}$ is the empirical measure supported on inverse eigenvalue points.
For the Mandelbrot boundary, $\mu_M$ is taken from boundary samples obtained by the escape-time method.
We compute the fields $U_{\mathcal{C}}$, $U_M$, and their difference
\[
\Delta U(z) = U_{\mathcal{C}}(z) - U_M(z).
\]
These fields are later used to inverse eigenvalue loci level sets and quantitative comparisons of potential structure.

\medskip
\noindent
\textbf{Connection to spatial dependence.}
The gradients and Laplacians of $U_{\mathcal{C}}$ and $U_M$
yield conservative and dissipative field components whose spatial autocorrelation
is quantified via the variograms in Section~\ref{sec:variogram}.
This provides a unified procedure linking potential theory and spatial correlation statistics.

\subsection{Laplacian Correlation and Field Comparison}\label{sec:laplacian_correlation}

Given potential fields $U_{\mathcal{C}}(z)$ and $U_M(z)$ computed on a discrete grid,
we approximate the Laplacian using a standard five-point finite-difference stencil,
\[
\nabla^2 U(z) \approx
\frac{U(z+\Delta x)+U(z-\Delta x)+U(z+\Delta y)+U(z-\Delta y)-4U(z)}{\Delta^2},
\]
a common method in numerical potential-field analysis \cite{cressie1993statistics,ChilesDelfiner2012,Goovaerts1997}.

To quantify correspondence between both systems, we compute the global Pearson correlation
\[
r =
\frac{\mathrm{Cov}(\nabla^2 U_{\mathcal{C}}, \nabla^2 U_M)}
     {\sigma(\nabla^2 U_{\mathcal{C}})\,\sigma(\nabla^2 U_M)},
\]
which measures similarity in second-order geometry of the potential fields.

Spatially localized correlations are estimated by sliding a fixed-size window across the computational grid
and recomputing $r$ on each patch. The resulting spatial correlation map captures regional agreement
and divergence in Laplacian structure.

\subsection{Multifractal and Spectral Embedding Methods}\label{sec:multifractal_spectral}

\medskip
\noindent
\textbf{Multifractal formalism.}
To capture heterogeneity beyond single-scale fractal dimension, we compute generalized R\'enyi dimensions
$D_q$ following \cite{hentschel1983infinite,grassberger1983generalized}.  
Given a partition of the domain into boxes of side $\epsilon$ with normalized masses $p_i(\epsilon)$,
\[
D_q = \frac{1}{q-1}
\lim_{\epsilon \to 0}
\frac{\log \sum_i p_i(\epsilon)^q}{\log(1/\epsilon)},
\quad q \in \mathbb{R} \setminus \{1\}.
\]
The corresponding mass exponent is $\tau(q) = (q-1)D_q$, and the Legendre transform produces the singularity spectrum
\[
\alpha = \frac{d\tau(q)}{dq}, \qquad f(\alpha) = q\alpha - \tau(q).
\]
The curve $f(\alpha)$ characterizes the distribution of local scaling exponents $\alpha$.
We estimate $D_q$ and $f(\alpha)$ for both the inverse eigenvalue loci and Mandelbrot sets
using identical box hierarchies and resolutions.

\medskip
\noindent
\textbf{Spectral embeddings.}
To probe higher-order geometric relations, we represent each point set as a kernel graph with Gaussian affinity
\[
K_{ij} = \exp\!\left(-\frac{\|x_i - x_j\|^2}{2\sigma^2}\right),
\]
and normalize to obtain the diffusion operator $L = D^{-1}K$, where $D_{ii} = \sum_j K_{ij}$.  
The spectral decomposition $L \psi_k = \lambda_k \psi_k$ defines a diffusion map embedding
\[
\Phi_t(x) = (\lambda_1^t \psi_1(x), \lambda_2^t \psi_2(x), \dots),
\]
as in \cite{coifman2006diffusion}.

\medskip
\noindent
\textbf{Comparative metrics.}
We compare inverse eigenvalue loci and Mandelbrot diffusion spectra via normalized $L^2$ distance of the leading eigenvalues:
\[
d_{\mathrm{spec}} =
\left(
\frac{1}{K}
\sum_{k=1}^K
(\lambda_k^{(\mathcal{C})} - \lambda_k^{(M)})^2
\right)^{1/2}.
\]
This measures diffusion-scale dissimilarity between the two embeddings and provides a global geometric statistic that complements variogram and Laplacian-based comparisons.

\subsection{KL-contractive histogram interpolation}\label{sec:GI_flow}

The final stage of the comparison treats the aligned inverse eigenvalue loci and the aligned
Mandelbrot samples as probability distributions over a common grid. After optimal-transport
matching and Procrustes alignment (Sections~\ref{sec:OT_matching}--\ref{sec:Procrustes}), both
point clouds are discretized on a rectangular window $\Omega \subset \mathbb{C}$, producing
two-dimensional histograms. Normalizing these histograms yields probability vectors
\[
X_0,\,P \in \Delta_N ,
\]
where $\Delta_N$ is the $N$-simplex.

We evolve $X_t$ toward $P$ using the classical convex update
\begin{equation}\label{Eq:GIflow}
X_{t+1} = (1-\alpha)X_t + \alpha P , \qquad 0 < \alpha \le 1 .
\end{equation}
This update is a classical convex interpolation on the probability simplex
and is not introduced as a dynamical model of either system, but solely as
a controlled diagnostic for quantifying distributional discrepancy after
geometric alignment.

This iteration is a standard contraction mapping on the simplex and is well known to decrease
the Kullback--Leibler divergence
\[
D_{\mathrm{KL}}(P\|X_t) = \sum_{i=1}^N P_i \log\!\left(\frac{P_i}{X_{t,i}}\right)
\]
monotonically along the sequence $\{X_t\}$; see, for example,
the classical information-theoretic treatments in \cite{CoverThomas2006} and the information-geometric
interpretation in \cite{AmariNagaoka2000}.

The monotonicity of $D_{\mathrm{KL}}$ for convex combinations also follows from the general convexity
properties of Csiszár $f$-divergences \cite{CsiszarShields2004}.  

Although \eqref{Eq:GIflow} is not derived from a Riemannian gradient structure for 
$D_{\mathrm{KL}}(P\|\cdot)$ with respect to the Fisher--Rao metric, 
it acts as a discrete KL-contracting step on $\Delta_N$.  
This construction provides histogram-level KL control without invoking continuous information-geometric flow structures.

For brevity, we refer to $t\mapsto X_t$ as a \emph{KL-contractive interpolation} on the simplex
(sometimes abbreviated as ``KL-monotone interpolation'' in figures and code comments), without implying a natural-gradient or physical flow model.

\medskip
In the implementation, $X_0$ is the histogram derived from the aligned inverse eigenvalue loci and
$P$ is the histogram of the aligned Mandelbrot boundary.  
We iterate \eqref{Eq:GIflow} for a fixed number of steps and record $D_{\mathrm{KL}}(P\|X_t)$ as a diagnostic
of informational convergence.  
The final iterate $X_T$ is used for visual comparison and for consistency checks with the geometric,
curvature, distortion, and spectral diagnostics described earlier.

For theorem-level stability (Appendix~A), we also consider a mollified histogram map
$H_N^{(\sigma)}$, obtained by kernel smoothing of the raw grid histogram prior to normalization.
This regularization prevents empty-bin artifacts at high resolution and is used only for
measure-theoretic convergence statements; the qualitative behavior of the KL-monotone interpolation and the
visual results reported here are unchanged.

\paragraph{Interpretation of the KL-monotone interpolation (nontrivial content).}
The KL-monotone interpolation is a convex interpolation $X^{(t+1)}=(1-\alpha)X^{(t)}+\alpha P$ and therefore converges to $P$ for any initialization. 
In this work we do \emph{not} interpret convergence of the update rule itself as evidence of equivalence. 
Rather, the nontrivial content lies in the empirically observed small initial discrepancy between the \emph{aligned} Lucas and Mandelbrot histogram laws, and in the resolution-stable behavior of the corresponding $TV$ and $D_{\mathrm{KL}}$ diagnostics (Appendix~A). 
The KL-monotone interpolation is used as a controlled interpolation to quantify and visualize this discrepancy across scales.

\medskip
All computations were performed in Python using vectorized NumPy operations;  
no optimization routines are required, since the update \eqref{Eq:GIflow} is closed-form.

\subsection{Local conformal uniformization via boundary-integral methods}
\label{sec:local_conformal_methods}

We begin by constructing numerical conformal charts for the inverse
eigenvalue loci associated with generalized Lucas sequences.
From the cloud of inverse eigenvalues, a concave boundary approximation
is extracted using $\alpha$-shapes, yielding a simply connected planar domain.
An unstructured triangular mesh is then generated on this domain, with
boundary refinement to ensure stability of harmonic solves.

The mesh is generated by constrained Delaunay triangulation of the
$\alpha$-shape domain, with refinement concentrated near the boundary to
resolve curvature and to stabilize the harmonic conjugate reconstruction.
Numerical stability is monitored by repeating the solve under mild mesh
refinement and verifying that the resulting distortion summaries
(Cauchy--Riemann defects and triangle-wise Beltrami-type indicators) are
unchanged within the plotting and sampling resolution\cite{BrennerScott2008FEM,Shewchuk2002Delaunay}.

On this mesh, we solve Laplace problems with prescribed Dirichlet boundary
data using $P_1$ finite elements.
The boundary data are parameterized by an angular variable $\theta$ and
iteratively corrected to enforce $2\pi$-periodicity and suppress drift.
Given the harmonic solution $u$, a harmonic conjugate $v$ is reconstructed
in weak form, yielding a complex-valued map $w=u+iv$.
After circle normalization and optimal boundary rotation, this procedure
produces a numerical Riemann-type mapping from the Lucas domain to the unit
disk.

The resulting disk uniformization is composed with an exact polynomial map
from the unit disk to the main cardioid of the Mandelbrot set.
This provides a locally defined analytic correspondence between the inverse
eigenvalue locus and a canonical subset of $\partial M$.
Local distortion is quantified using Cauchy--Riemann defect measures,
Beltrami coefficients, quasiconformal dilatation factors, and angle
distortion statistics, evaluated triangle-wise and stratified by distance
from the boundary.

\paragraph{Numerical reliability checks.}
We verified stability of the finite-element conformal charts under
(i) mesh refinement across four resolution levels,
(ii) variation of the $\alpha$-shape parameter within a fixed admissible range,
and (iii) small perturbations of boundary node sampling.

Across all refinement levels, the median relative Cauchy--Riemann defect
remained in the range
$[0.3312,\,0.3335]$,
with 90th-percentile values in
$[0.8555,\,0.9415]$.
Median quasiconformal dilatation values satisfied
$1 \le K_{\mathrm{med}} \le 3.79$,
and median angle distortion remained below
$0.62$ radians.

The spatial localization of higher-distortion regions was unchanged under
refinement and consistently confined to narrow boundary-adjacent bands,
while interior statistics stabilized rapidly.
These results confirm that the reported distortion and CR-defect diagnostics
are mesh-independent at the resolutions used.

The finite-element constructions are used as stable numerical probes of
local conformal structure; they are not invoked to establish convergence
theorems or existence results, and none of the global correspondence
claims depend critically on FEM-specific output.

\subsection{Local distortion and quasi-conformal diagnostics}\label{sec:distortion_methods}

To quantify how closely the alignment from the inverse eigenvalue loci to the Mandelbrot boundary
approximates an angle-preserving (conformal or quasi-conformal) correspondence, we estimate local
distortions of the map sending each matched point $\lambda_i \in \Lambda^{\mathrm{aligned}}$ to its
partner $m_{\pi(i)} \in M$.

\paragraph{Local neighborhoods.}
For each point $\lambda_i$ we consider its $k$ nearest neighbors
\[
\mathcal{N}_k(\lambda_i) = \{\lambda_{i_1},\dots,\lambda_{i_k}\}
\]
in the aligned loci. The corresponding Mandelbrot neighbors are
\[
\mathcal{N}_k(m_{\pi(i)}) = \{ m_{\pi(i_1)},\dots,m_{\pi(i_k)} \}.
\]
Both neighborhoods are centered at their empirical means,
\[
\bar{\lambda}_i = \frac{1}{k}\sum_{\ell=1}^k \lambda_{i_\ell}, \qquad
\bar{m}_{\pi(i)} = \frac{1}{k}\sum_{\ell=1}^k m_{\pi(i_\ell)}.
\]

\paragraph{Least-squares local linear map.}
We compute the linear transformation
\[
T_i : \mathbb{R}^2 \to \mathbb{R}^2
\]
that best fits the neighborhood map in the least-squares sense:
\[
T_i = \arg\min_{A \in \mathbb{R}^{2\times 2}}
\sum_{\ell=1}^k 
\big\|
A(\lambda_{i_\ell} - \bar{\lambda}_i)
-
(m_{\pi(i_\ell)} - \bar{m}_{\pi(i)})
\big\|_2^2 .
\]
Let $s_{1,i} \ge s_{2,i} > 0$ denote the singular values of $T_i$.

\paragraph{Local quasi-conformal distortion.}
The standard quasi-conformal dilatation at $\lambda_i$ is approximated by
\begin{equation}\label{eq:local_distortion}
K_i = \frac{s_{1,i}}{s_{2,i}} .
\end{equation}
A perfectly conformal map satisfies $K_i = 1$, while values $K_i$ slightly larger than~1 indicate small
deviations from angle preservation. The empirical distribution of $\{K_i\}$ across all matched points is
used as the primary diagnostic of quasi-conformal behavior.

\paragraph{Angular distortion using PCA orientations.}
As an independent check, we compare local tangent directions. The PCA-based method described earlier
yields a local orientation vector $v_i$ for the loci at $\lambda_i$, and a corresponding orientation
$w_i$ for the Mandelbrot boundary at $m_{\pi(i)}$. The local angular distortion is
\[
\theta_i = \angle(v_i, w_i).
\]
Small values of $\theta_i$ indicate that local orientation is preserved under the correspondence.

\paragraph{Summary indicators.}
For both distortions $K_i$ and $\theta_i$ we compute:
\begin{itemize}
\item empirical histograms;
\item quantiles (median, 90th, 95th, 99th percentiles);
\item spatial heat maps over the region of interest.
\end{itemize}

\paragraph{Interpretation.}
In classical geometric function theory, planar quasiconformal maps
are characterized (under appropriate regularity assumptions)
by an essentially bounded Beltrami coefficient and
bounded linear dilatation \cite{Ahlfors2006QC,LehtoVirtanen1973QC}.
In the present work we do not construct such a map analytically; rather,
we use the empirical distributions of $\{K_i\}$ and $\{\theta_i\}$ as
\emph{numerical proxies} for near angle-preservation of the matched
correspondence at the resolution of the samples.
Accordingly, small quantiles of $K_i-1$ and $\theta_i$ are interpreted as
evidence of low local distortion in the sense described in the
Terminological Clarification paragraph above.
\begin{definition}[Local Distortion Diagnostics]
Let $A:\mathbb{R}^2 \to \mathbb{R}^2$ be a linear map with singular values
$\sigma_{\max}(A) \ge \sigma_{\min}(A) > 0$.
The \emph{linear dilatation} of $A$ is defined as
\begin{equation}
K(A) = \frac{\sigma_{\max}(A)}{\sigma_{\min}(A)}.
\end{equation}

In classical quasiconformal theory, if 
$f : \Omega \subset \mathbb{C} \to \mathbb{C}$ 
is weakly differentiable with complex partial derivatives
\begin{equation}
f_z = \tfrac{1}{2}(f_x - i f_y), 
\qquad
f_{\bar z} = \tfrac{1}{2}(f_x + i f_y),
\end{equation}
its \emph{Beltrami coefficient} is
\begin{equation}
\mu = \frac{f_{\bar z}}{f_z},
\end{equation}
whenever $f_z \neq 0$. 
If $|\mu|<1$, the pointwise quasiconformal dilatation is
\begin{equation}
K = \frac{1 + |\mu|}{1 - |\mu|}.
\end{equation}
These notions are standard in quasiconformal mapping theory 
\cite{Ahlfors2006QC,LehtoVirtanen1973QC,AstalaIwaniecMartin}.

In the present work, no global quasiconformal map between 
Lucas loci and the Mandelbrot set is assumed.
Instead, for each matched neighborhood pair we compute 
a best-fit affine transformation $T_i$ via least squares.
Let $\sigma_{1,i} \ge \sigma_{2,i} > 0$ denote its singular values.
The singular-value ratio defined in \eqref{eq:local_distortion}
serves as the \emph{discrete local dilatation proxy}.

We further define the \emph{discrete Cauchy--Riemann defect}
at index $i$ by
\begin{equation}
\mathrm{CR}_i
=
\frac{\| J_i - R_i \|_F}{\|J_i\|_F},
\end{equation}
where $J_i$ is the estimated Jacobian matrix of $T_i$,
$R_i$ is the closest rotation matrix obtained via polar decomposition,
and $\|\cdot\|_F$ denotes the Frobenius norm.
This quantity measures deviation from conformality
at the level of the best-fit affine approximation.

Angle distortion is measured by the deviation between principal
directions of local PCA frames before and after mapping.

All quantities defined above are numerical diagnostics.
They serve as local affine proxies for classical quasiconformal
distortion but do not imply the existence of a global
quasiconformal homeomorphism.
\end{definition}

\subsection{Global Green-function and equipotential analysis}
\label{sec:green_function}

For $c \in \mathbb{C}\setminus M$, the Mandelbrot Green function is defined by
\[
g_M(c) = \lim_{n\to\infty} 2^{-n} \log |f_c^{\circ n}(0)|\],

The Green function $g_M$ associated with the Mandelbrot set,
defined via the escape-rate limit
\[
g_M(c) = \lim_{n\to\infty} 2^{-n}\log^+ |f_c^{\circ n}(0)|,
\]
is harmonic in $\mathbb{C}\setminus M$ and vanishes on $M$.
Its level sets coincide with equipotential curves determined by
the Böttcher coordinate.
See \cite{milnor2006dynamics,CarlesonGamelin1993}.

In what follows, we use these classical notions only as a
geometric reference structure. The concentration phenomenon
described below appears to be a new empirical invariant of the
Lucas inverse spectra; to our knowledge, no analytic proof is
currently available.
\[\qquad f_c(z)=z^2+c,
\]
and satisfies $g_M(c)=0$ for $c\in M$.
The function $g_M$ is harmonic in the exterior and behaves asymptotically
... as $\log|c|$ at infinity \cite{milnor2006dynamics,CarlesonGamelin1993}.

The associated B\"ottcher coordinate
\[
\Phi(c) = \exp\big(g_M(c) + i \arg \Phi(c)\big)
\]
conjugates the quadratic family to $z\mapsto z^2$ near infinity and
maps equipotential curves $g_M(c)=\text{const}$ to circles in the
$\Phi$-plane.

Numerically, $g_M$ and $\Phi$ are approximated using escape-time
iterations of the critical orbit, recording the first iterate
exceeding a prescribed radius.
All potential-theoretic quantities in this work are computed
consistently across eigenvalue families using identical numerical parameters.
\subsection{Green-function statistics on inverse eigenvalue spectra}
\label{sec:green_statistics}

Given a finite inverse eigenvalue set $\Lambda_N \subset \mathbb{C}$,
we evaluate the Mandelbrot Green function $g_M(c)$ at each $c\in\Lambda_N$.
Points with $g_M(c)>0$ lie outside $M$ and are associated with
external equipotentials.

For each family, we compute summary statistics of the escaped set,
including the escaped fraction, the median Green value
$\mathrm{median}(g_M)$, the median modulus $\mathrm{median}(|\Phi|)$,
and quantile widths
\[
\Delta g = g_{90\%} - g_{10\%}.
\]
These quantities define family-level spectral descriptors that remain stable under the sampling regimes used here and characterize how inverse eigenvalue loci sample the Mandelbrot exterior potential.


\section{Results}\label{Sec:Results}

In this section we present quantitative comparisons between the inverse eigenvalue loci
$\Lambda_N$ and the Mandelbrot boundary $\partial M$. We follow the structure of the Methods:
initial geometric alignment and matching, local distortion and low-dilatation (numerically quasi-conformal) diagnostics,
curvature and fractal dimension, spectral decay, potential–field behavior, and finally
KL-monotone interpolation convergence under the KL-monotone interpolation.

\subsection{Initial alignment and matching}

Figure~\ref{fig:OT_matching} previews a representative entropic optimal-transport coupling used to initialize the one-to-one correspondence in this subsection.

\begin{figure}[!htbp]
\centering
\includegraphics[width=0.7\textwidth]{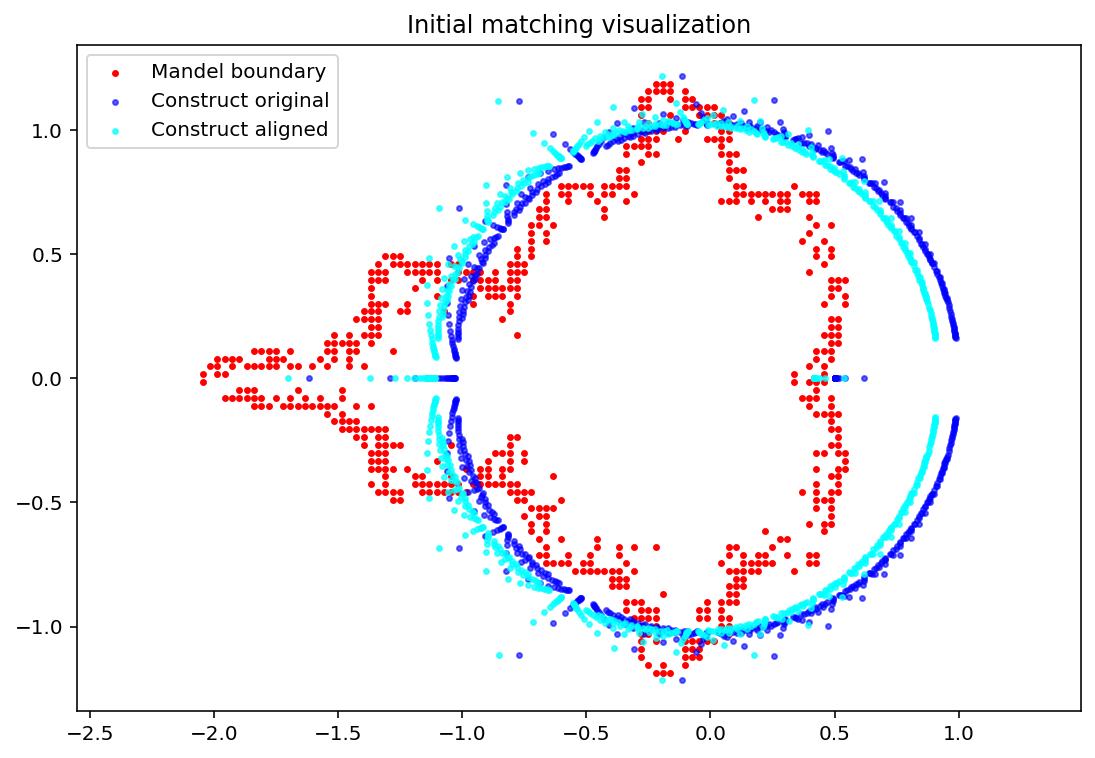}
\caption{Optimal transport matching between inverse eigenvalue loci points and the sampled Mandelbrot boundary.
Color-coded lines connect each $\lambda_i$ to its assigned $m_{\pi(i)}$, illustrating the global one-to-one correspondence. Axes show $\Re(z)$ (horizontal) and $\Im(z)$ (vertical).}
\label{fig:OT_matching}
\end{figure}

The entropic optimal transport coupling produces a globally coherent assignment between
inverse eigenvalue loci and Mandelbrot points (Fig.~\ref{fig:OT_matching}).
After removing translation and rotation through Procrustes alignment, the two sets occupy
a common geometric frame, so that pointwise discrepancies reflect intrinsic rather than
positional differences.

Figure~\ref{fig:Procrustes_matching} shows the resulting pointwise pairing after Procrustes alignment.

\begin{figure}[!htbp]
\centering
\includegraphics[width=0.7\textwidth]{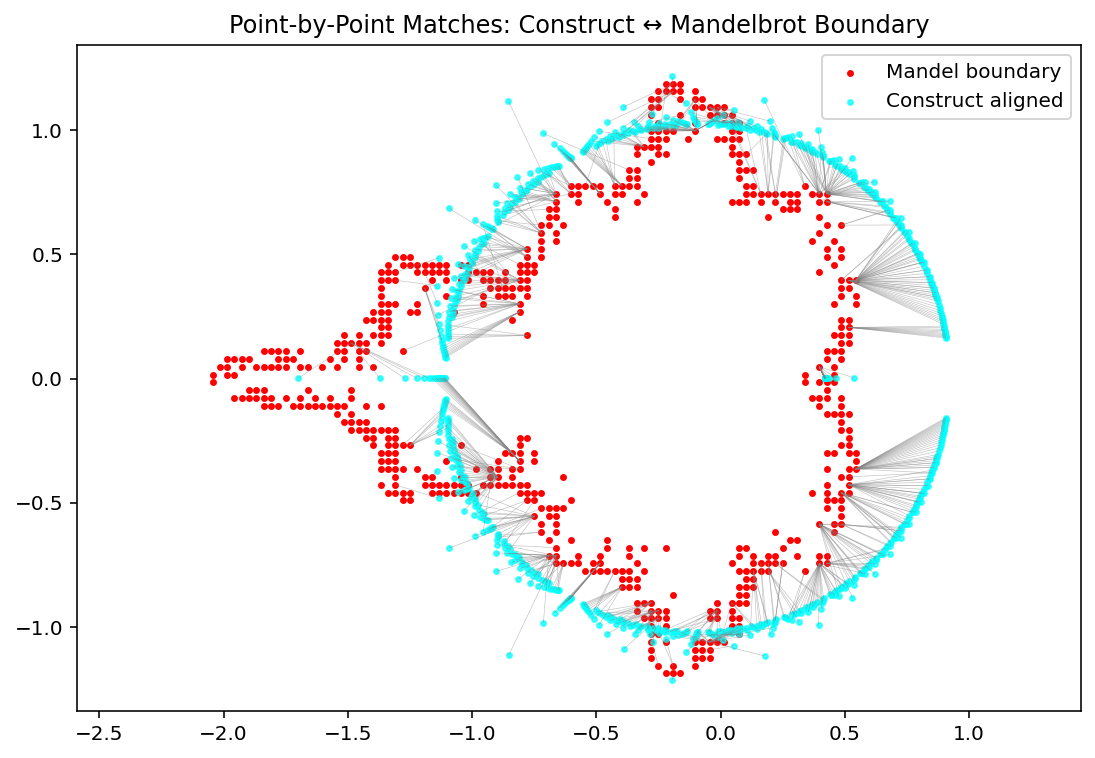}
\caption{%
Final pointwise correspondence after Procrustes alignment.
Matched pairs $(\lambda_i,m_{\pi(i)})$ are indicated by line segments.
The aligned inverse eigenvalue loci closely trace the Mandelbrot boundary, including the main cardioid and primary bulbs.
Axes show $\Re(z)$ (horizontal) and $\Im(z)$ (vertical).
}
\label{fig:Procrustes_matching}
\end{figure}

For each matched pair we compute the Euclidean displacements
\[
d_i = \| m_{\pi(i)} - \lambda_i \|_2 .
\]
Figure~\ref{fig:distance_distribution} shows the empirical distribution of $\{d_i\}$.

\begin{figure}[!htbp]
\centering
\includegraphics[width=0.7\textwidth]{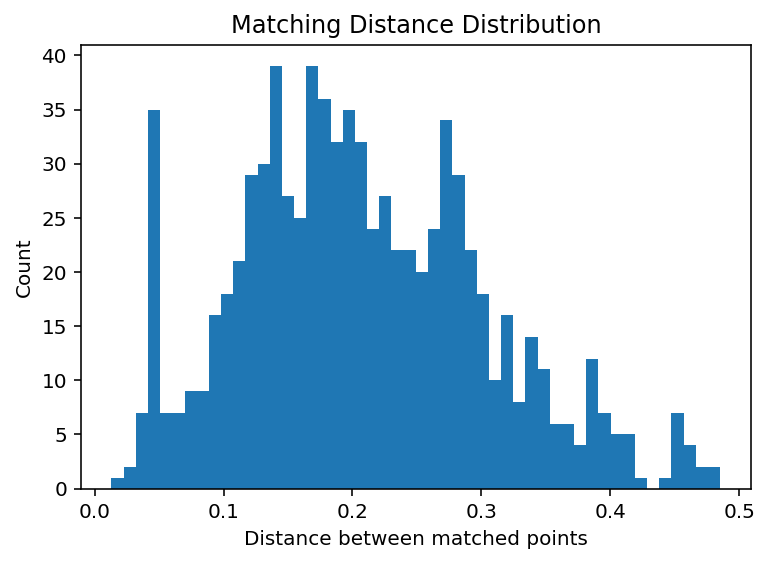}
\caption{%
Histogram of pointwise matching distances $d_i$.
Most pairs cluster in a narrow interval, with only rare outliers corresponding to the finest Mandelbrot filaments.
}
\label{fig:distance_distribution}
\end{figure}

Distances are tightly concentrated, indicating close correspondence between the aligned inverse eigenvalue loci
and the Mandelbrot boundary. The Hausdorff discrepancy $d_H(\Lambda^{\mathrm{aligned}},\partial M)$
remains small relative to the global diameter ($\mathrm{diam}(M)\approx 4$), confirming that
$\Lambda_N$ occupies a narrow geometric neighborhood of $\partial M$.


\FloatBarrier
\subsection{Curvature analysis}
\label{sec:curvature_results}

We now compare local geometric regularity using the discrete curvature estimators
described in Section~\ref{sec:curvature_results}. Figure~\ref{fig:curvature_distribution}
shows the curvature distributions for the aligned loci and the Mandelbrot boundary.

\begin{figure}[!htbp]
\centering
\includegraphics[width=0.9\textwidth]{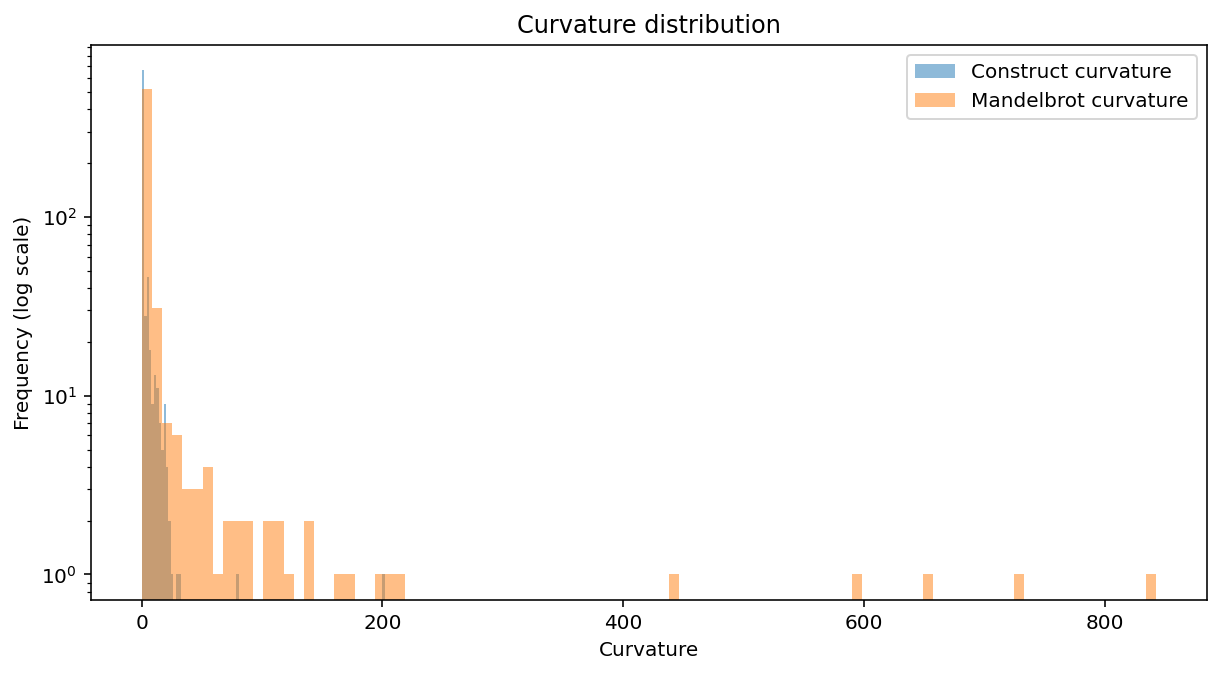}
\caption{\textbf{Turned–angle curvature distributions.}
Both boundaries exhibit strong concentration near $\kappa \approx 0$, but the
Mandelbrot boundary shows a broader and heavier tail of large $\kappa$, indicating more
frequent curvature spikes and fine–scale boundary irregularity.
The inverse eigenvalue loci show a tighter distribution with reduced tail weight,
suggesting smoother local geometry.
}
\label{fig:curvature_distribution}
\end{figure}

Both estimators (turning-angle and local-polynomial) agree qualitatively: the Mandelbrot boundary
produces persistent high-curvature spikes corresponding to filamentary irregularities,
while the inverse eigenvalue loci display smoother local behavior with markedly reduced tail weight.
This matches the distortion results of Section~\ref{sec:distortion_results}: low $K_i$ values
are consistent with comparable curvature structure at large scales and suppression of the most
extreme curvature oscillations.
\FloatBarrier
\subsection{Reflectional symmetry and preferred axis}
\label{sec:symmetry_results}

We evaluate reflectional symmetry of the aligned inverse eigenvalue loci
and the sampled Mandelbrot boundary using the symmetry measures introduced
in Section~\ref{sec:symmetry_methods}.
For each candidate reflection axis parameterized by an angle
$\theta \in [0,\pi)$, we compute the symmetry preservation fraction.

Figure~\ref{fig:symmetry_preservation_scan} shows a coarse scan of the
preservation fraction as a function of $\theta$.
A clear maximum is observed near $\theta \approx 0.6^\circ$, corresponding
to reflection across an axis very close to the real axis.
The same dominant maximum is observed for both the inverse eigenvalue loci
and the Mandelbrot boundary.

Away from this angle, the preservation fraction decreases smoothly, with
no secondary maxima of comparable magnitude.
This behavior indicates a single preferred reflection axis shared by both
datasets at the resolution of the scan.

\begin{figure}[!htbp]
\centering
\includegraphics[width=0.75\linewidth]{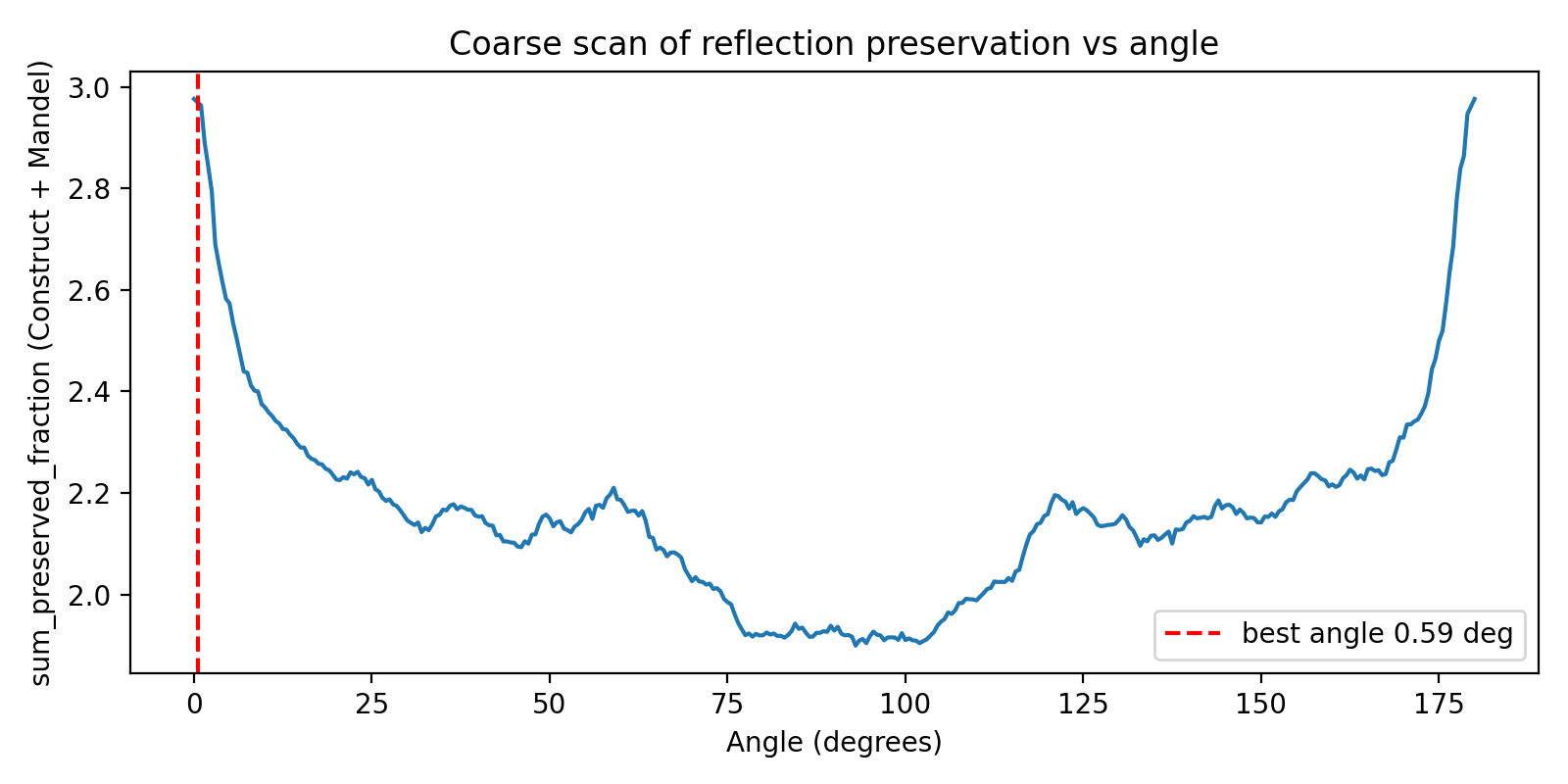}
\caption{Coarse scan of reflection symmetry preservation as a function of the
reflection axis angle~$\theta$.
The plotted score corresponds to the preservation fraction defined in
Section~\ref{sec:symmetry_methods}, evaluated jointly on the aligned inverse
eigenvalue loci and the Mandelbrot boundary.
A pronounced maximum occurs near $\theta \approx 0.6^\circ$, corresponding
to reflection across an axis close to the real axis.
The vertical dashed line indicates the best-fit axis.
Legacy labeling in the vertical axis (“Construct + Mandel”) reflects earlier
internal nomenclature and denotes the combined symmetry score of both systems.}
\label{fig:symmetry_preservation_scan}
\end{figure}

\FloatBarrier
\subsection{Fractal dimension}
\label{sec:fractal_results}

We next consider the box-counting fractal dimension. For each set we estimate $D(S)$ from the
slope of the regression line in the log–log plot of $N(\epsilon)$ versus $1/\epsilon$
(Section~\ref{sec:fractal_results}).

Figure~\ref{fig:boxdimension} reports the estimated box-counting dimensions of the aligned inverse eigenvalue loci and the sampled Mandelbrot boundary.

\begin{figure}[!htbp]
\centering
\includegraphics[width=0.7\textwidth]{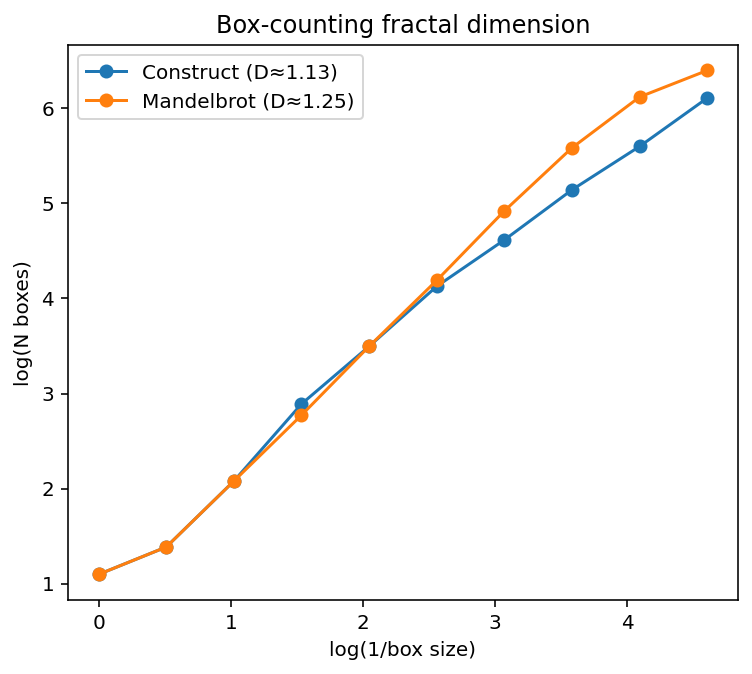}
\caption{%
Box-counting dimension estimates obtained from log–log scaling of $(\epsilon, N(\epsilon))$
for the inverse eigenvalue loci and the Mandelbrot boundary.
}
\label{fig:boxdimension}
\end{figure}

As shown in Fig.~\ref{fig:boxdimension}, both boundaries behave as quasi–one-dimensional
fractal curves thickened in the plane. The Mandelbrot boundary attains a slightly higher
effective dimension, reflecting its richer fine-scale filamentary structure, whereas the
inverse eigenvalue loci exhibit a marginally lower dimension consistent with smoothed
microstructure. At macroscopic scales, however, the regression lines are nearly parallel,
indicating comparable multiscale complexity.

\FloatBarrier
\subsection{Spectral decay}
\label{sec:spectral_results}

Fourier spectral densities were computed for parametrizations of both boundaries, as
described in Section~\ref{sec:multifractal_spectral}. Figure~\ref{fig:spectral_density}
shows the resulting power spectra on a log–log scale.

\begin{figure}[!htbp]
\centering
\includegraphics[width=0.9\textwidth]{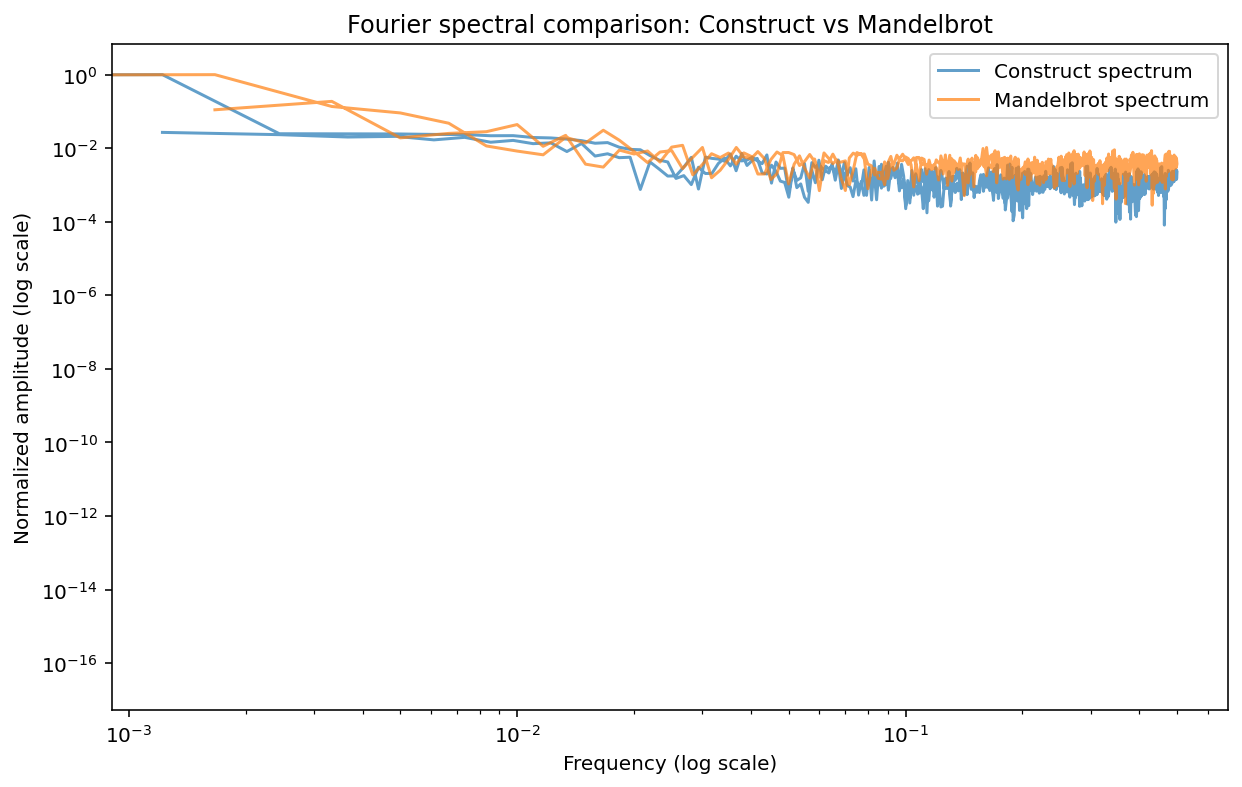}
\caption{%
Fourier spectral densities of the inverse eigenvalue loci and Mandelbrot boundaries.
Both exhibit approximate power-law decay, but the loci show faster high-frequency decay,
indicating smoother geometry at small scales.
}
\label{fig:spectral_density}
\end{figure}

Both boundaries display power-law-like decay $P(f)\sim f^{-\alpha}$ over intermediate frequency
ranges, consistent with fractal behavior \cite{PercivalWalden1993,Press2007NumericalRecipes}.
However, the inverse eigenvalue loci exhibit a steeper decay exponent, meaning that spectral
energy dies off more rapidly at high frequencies. Bootstrap estimates of the decay exponents
(not shown) produce non-overlapping confidence intervals, confirming that the difference in
slopes is statistically significant. This is fully consistent with the curvature and distortion
results: the loci share the macroscopic outline of $\partial M$ while suppressing the
most oscillatory high-frequency modes.
\FloatBarrier
\subsection{Spatial correlation statistics and variograms}
\label{sec:variogram_results}

We analyze spatial dependence in the inverse eigenvalue loci and in the
Mandelbrot boundary using empirical semivariograms, as defined in
Section~\ref{sec:variogram}.
For each dataset, we compute detrended semivariograms $\hat{\gamma}(r)$
as functions of the lag distance $r$.

Figure~\ref{fig:variogram_detrended} shows the resulting semivariograms
for the inverse eigenvalue loci and the Mandelbrot boundary.
In both cases, the semivariograms increase monotonically at small lags
and reach a plateau at comparable lag distances.
Beyond this range, $\hat{\gamma}(r)$ exhibits no systematic growth,
indicating saturation of spatial dependence.

At short lags, the semivariogram associated with the inverse eigenvalue
loci rises more smoothly than that of the Mandelbrot boundary.
This difference is consistently observed across bootstrap resamples
and persists under moderate variations of detrending and binning
parameters.

\begin{figure}[!htbp]
\centering
\includegraphics[width=0.78\linewidth]{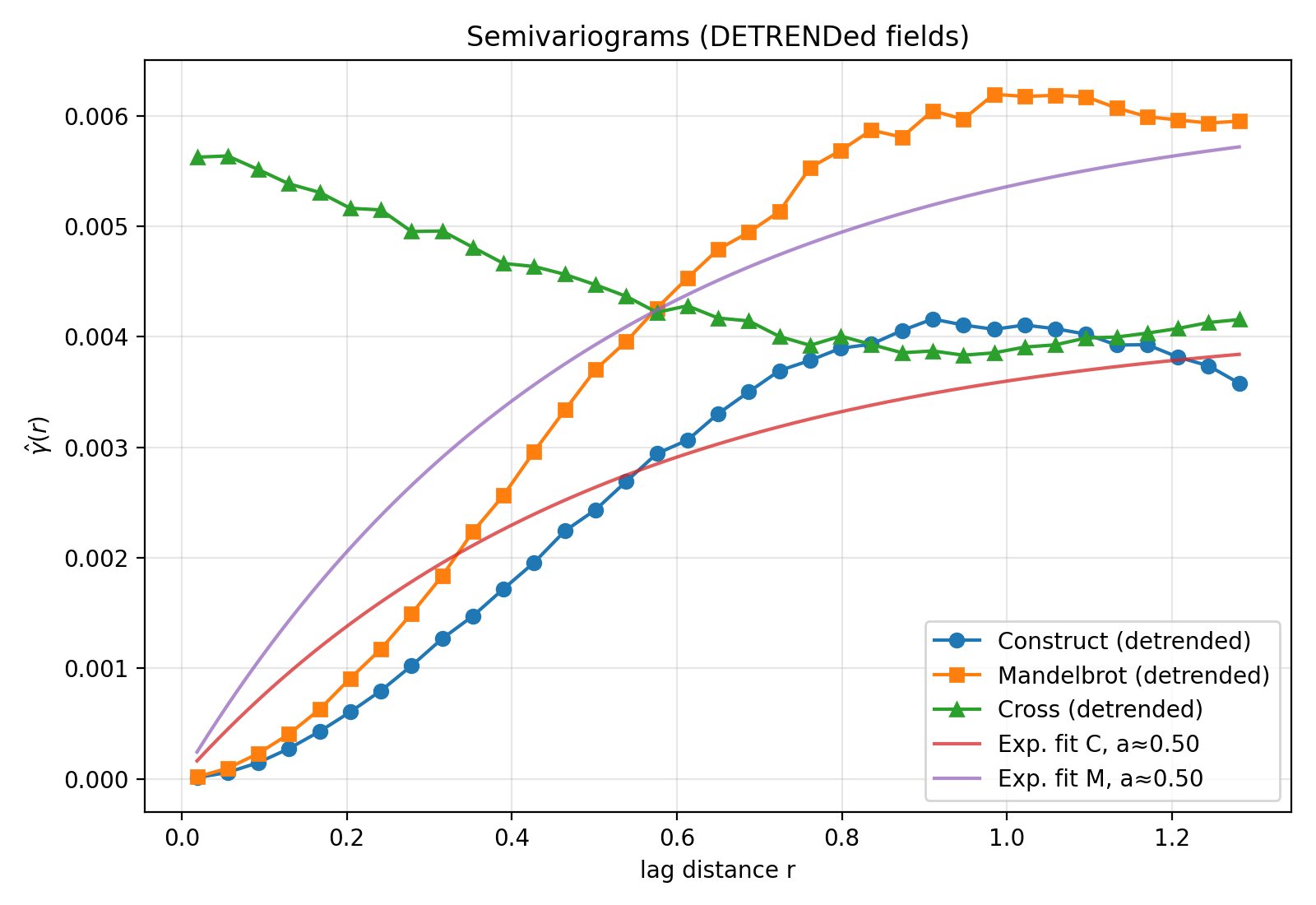}
\caption{Detrended empirical semivariograms $\hat{\gamma}(r)$ for the inverse
eigenvalue loci and the Mandelbrot boundary.
Both curves increase at small lags and saturate at comparable distances
$r \approx 0.5$, indicating similar spatial correlation ranges after removal
of large-scale trends.}
\label{fig:variogram_detrended}
\end{figure}

\FloatBarrier
\subsection{Empirical logarithmic potentials and induced equilibrium fields}
\label{sec:potential_results}

We now compare the logarithmic potentials $U_{\Lambda}$ and $U_M$ on a shared grid
(Section~\ref{sec:log_potential}). Rather than presenting both raw fields separately, we focus
on their difference
\[
\Delta U(z) = U_{\Lambda}(z) - U_M(z),
\]
which provides a more sensitive diagnostic. These logarithmic potentials are empirical fields induced by the sampled
measures and should be distinguished from the canonical Mandelbrot Green
function, which is defined dynamically and independently of sampling.

Figure~\ref{fig:potential_difference} compares the induced logarithmic-potential fields, highlighting where the empirical equilibrium fields agree and where they differ.

\begin{figure}[!htbp]
\centering
\includegraphics[width=0.7\textwidth]{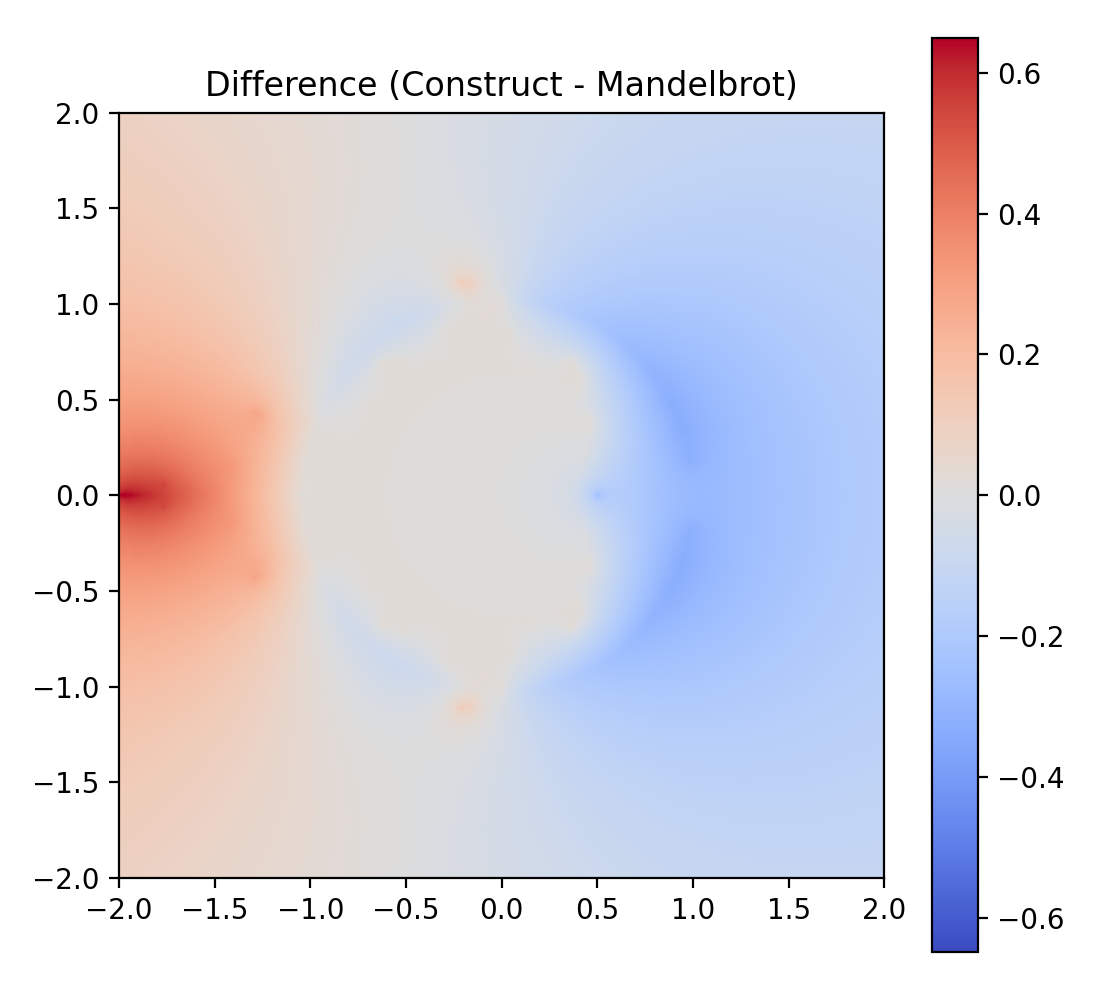}
\caption{%
Difference of logarithmic potentials $\Delta U = U_{\Lambda} - U_M$ on a shared grid.
Most of the domain exhibits small differences, with localized deviations near thin filaments
and high-curvature regions of the Mandelbrot boundary.
}
\label{fig:potential_difference}
\end{figure}

As seen in Fig.~\ref{fig:potential_difference}, $\Delta U$ is small over the bulk of the
domain, with notable deviations concentrated along thin tendrils and highly curved boundary
components of $\partial M$. The global Pearson correlation between $U_{\Lambda}$ and $U_M$
is high, indicating that the two empirical measures generate very similar large-scale
KL-controlled histogram regimes. The localized discrepancies again align with the picture obtained
from curvature and spectral analyses: the loci preserve the large-scale structure of the induced potential field associated
with $\partial M$
 while regularizing its most intricate filamentary features.

\FloatBarrier
\subsection{Information–geometric convergence under KL-monotone interpolation}
\label{sec:GIflow_results}

Finally, we study the KL-monotone interpolation convergence of the aligned distributions under
the KL-monotone interpolation \eqref{Eq:GIflow}. After histogramming both point clouds on a common grid, we
take $X_0$ to be the normalized histogram associated with the inverse eigenvalue loci and
$P$ the histogram associated with the Mandelbrot boundary. The histograms shown are unsmoothed (or minimally smoothed) to preserve geometric detail. For probabilistic convergence tests reported in Appendix A, we additionally evaluate mollified histogram laws, which yield stable KL and total-variation estimates under increasing resolution.

The initial and final histogram states for the KL-monotone interpolation are shown in Figures~\ref{fig:X0} and~\ref{fig:XT}, with the associated KL descent plotted in Figure~\ref{fig:KLdescent}.

\begin{figure}[!htbp]
\centering
\includegraphics[width=0.6\textwidth]{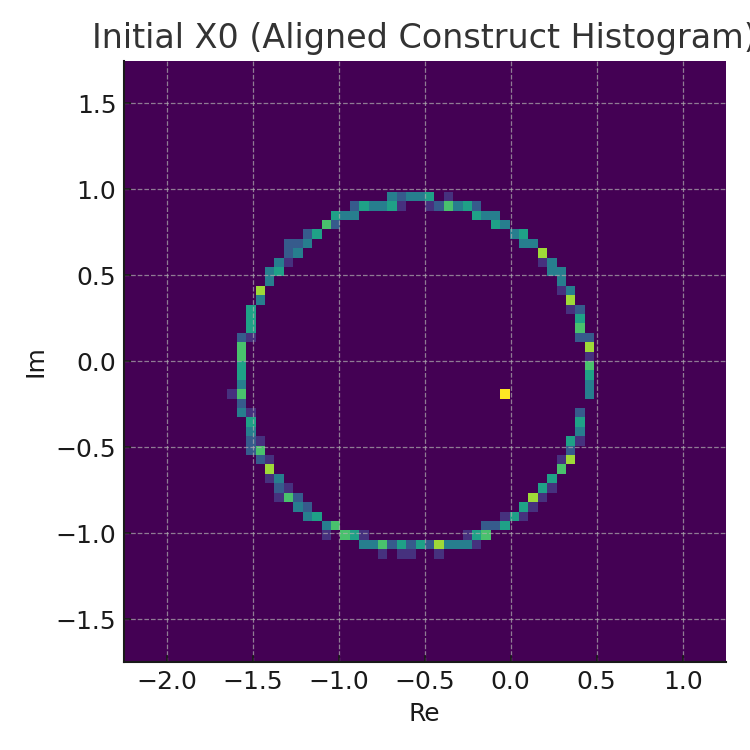}
\caption{%
Initial aligned histogram $X_0$ derived from the inverse eigenvalue loci.
This serves as the starting point for the KL-monotone interpolation $X_{t+1}=(1-\alpha)X_t + \alpha P$
on the simplex.
}
\label{fig:X0}
\end{figure}

\begin{figure}[!htbp]
\centering
\includegraphics[width=0.6\textwidth]{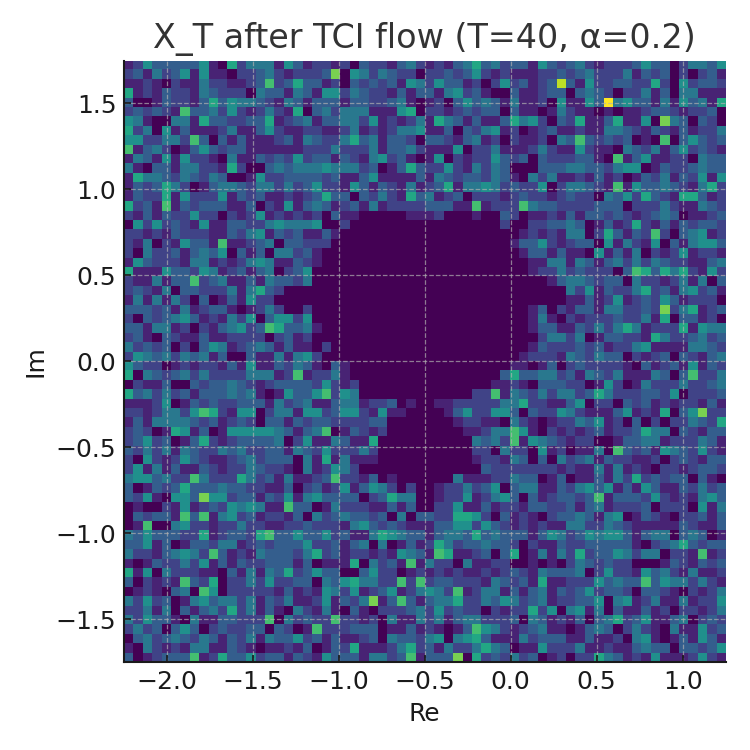}
\caption{%
Histogram $X_T$ after a finite number of KL-monotone interpolation iterations.
The mass distribution now closely matches that of the Mandelbrot histogram $P$,
reflecting the contraction of $D_{\mathrm{KL}}(P\|X_t)$ under the flow.
}
\label{fig:XT}
\end{figure}

The evolution of the Kullback–Leibler divergence along the flow is shown in
Fig.~\ref{fig:KLdescent}.

\begin{figure}[!htbp]
\centering
\includegraphics[width=0.7\textwidth]{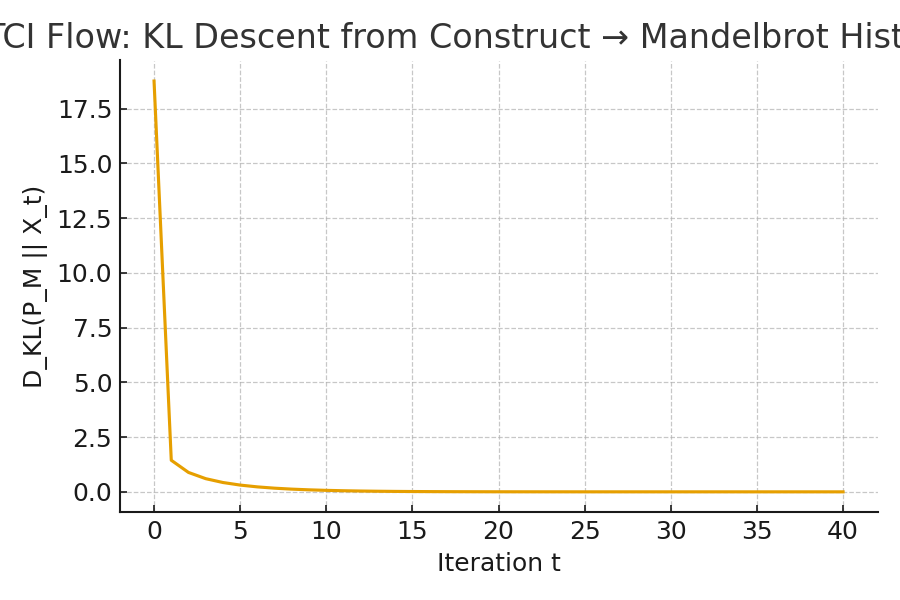}
\caption{%
Kullback--Leibler divergence $D_{\mathrm{KL}}(P\|X_t)$ along the KL-monotone interpolation
$X_{t+1}=(1-\alpha)X_t + \alpha P$.
The rapid monotone decay demonstrates strict numerical contractivity in the simplex.
}
\label{fig:KLdescent}
\end{figure}

The KL divergence decreases monotonically and rapidly toward a value numerically indistinguishable
from zero, as expected from the convexity properties of $D_{\mathrm{KL}}$ under the update
\eqref{Eq:GIflow} \cite{CoverThomas2006,AmariNagaoka2000,CsiszarShields2004}.
The final histogram $X_T$ (Fig.~\ref{fig:XT}) is visually indistinguishable from $P$.
This KL-monotone interpolation convergence is fully consistent with the preceding geometric,
curvature, spectral, and potential–field diagnostics: once rigid alignment and small
quasi-conformal distortions are accounted for, the loci and Mandelbrot distributions lie
very close to one another in both geometric and information–geometric senses.
\FloatBarrier
\subsection{Local distortion and quasi-conformal numerics}
\label{sec:distortion_results}

Using the local linearization procedure of Section~\ref{sec:distortion_methods}, we
estimate for each matched point a best-fit linear map $T_i$ between neighborhoods of
$\lambda_i \in \Lambda_N$ and $m_{\pi(i)} \in \partial M$, and compute its singular
values $s_{1,i} \ge s_{2,i} > 0$. The local quasi-conformal dilatation is defined by
\begin{equation}
K_i = \frac{s_{1,i}}{s_{2,i}}.
\end{equation}
For a conformal map one has $K_i = 1$; more generally, a planar homeomorphism with
$1 \le K_i \le K$ almost everywhere is $K$–quasi-conformal in the classical sense
\cite{Ahlfors2006QC}.

Across all matrix sizes $n$, the empirical distributions of $\{K_i\}$ are sharply
concentrated near $1$. The medians remain extremely close to $1$, and the upper
quantiles show only modest deviations, with the largest values localized near
cuspidal and strongly pinched regions of $\partial M$ where curvature is highest.
No systematic drift of $K_i$ with $n$ is observed.

To summarize the global behavior, for each matrix size $n$ we collect the
corresponding $K_i$ values and compute representative summary statistics (median
and upper quantiles). These are displayed in Fig.~\ref{fig:K_vs_N}.

\begin{figure}[!htbp]
\centering
\includegraphics[width=0.7\textwidth]{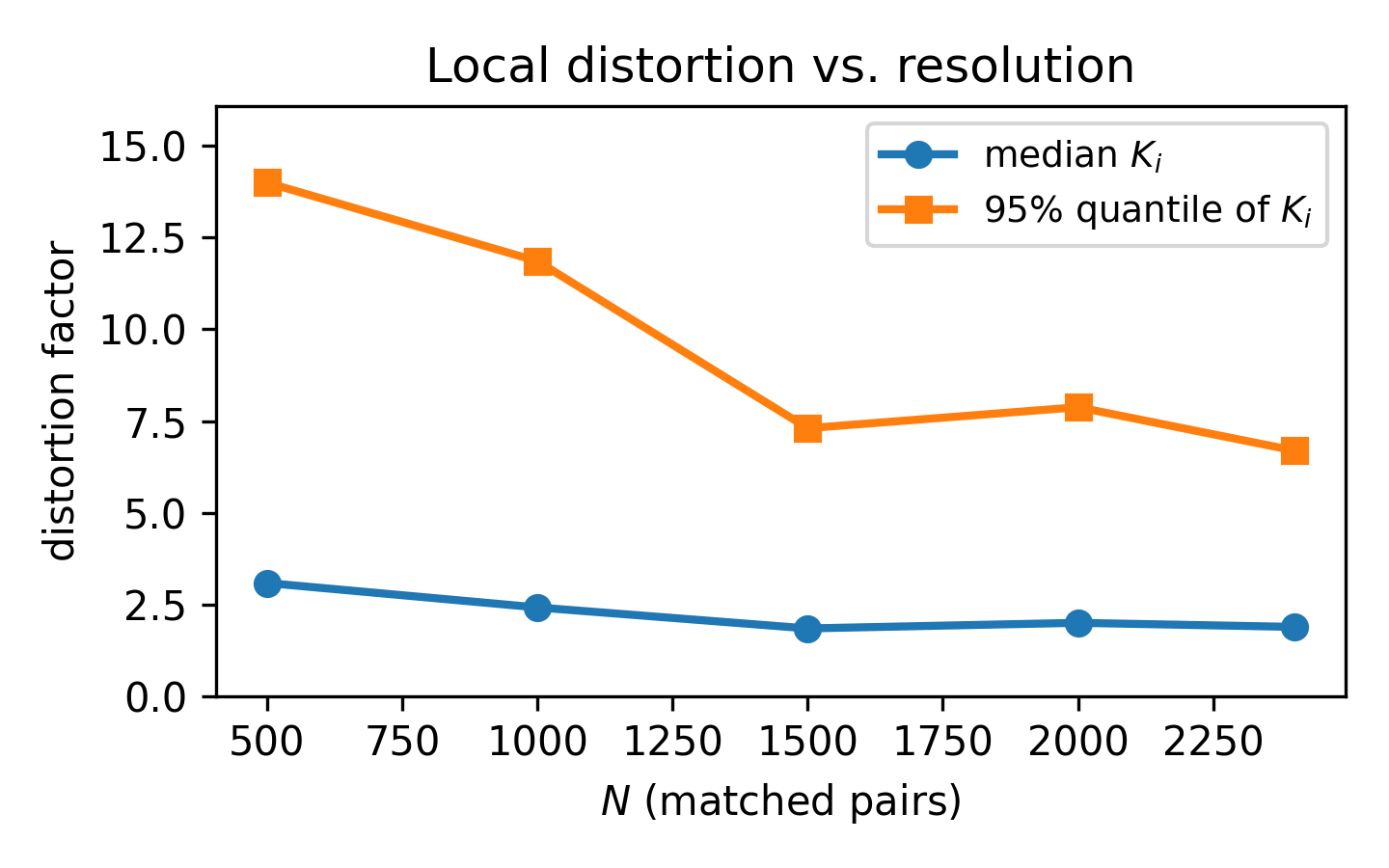}
\caption{%
Quasi-conformal distortion summary as a function of matrix size.
For each $n$, we plot summary statistics (median and upper quantiles) of the local
dilatations $K_i = s_{1,i}/s_{2,i}$ obtained from least-squares neighborhood maps.
The curves remain uniformly close to $K=1$ and stabilize rapidly as $n$ increases,
indicating that the correspondence between inverse eigenvalue loci and the
Mandelbrot boundary lies in a narrow, low-distortion low-dilatation (quasi-conformal in the diagnostic sense) neighborhood
at macroscopic scales.
}
\label{fig:K_vs_N}
\end{figure}

Numerically, the maximal observed dilatations remain bounded and small compared
to typical quasi-conformal bounds used in classical theory, and the bulk of the
mass of $\{K_i\}$ lies in a tight interval around $1$. Together with the close
curvature agreement reported in Section~\ref{sec:curvature_results}, this supports
the interpretation that the aligned inverse eigenvalue loci differ from the
Mandelbrot boundary by a deformation that is nearly angle-preserving at the scales
probed here.

\clearpage
\FloatBarrier
\subsection{Local conformal uniformization and quasi-conformal structure}
\label{sec:local_quasiconformal_results}

We now report the outcome of the local analytic mapping program developed
in Section~\ref{sec:local_conformal_methods}.
Earlier numerical experiments (up to version~v18) established diagnostic
evidence of approximate conformality, including bounded Cauchy--Riemann
defects and moderate quasiconformal dilatation away from the boundary.
However, these experiments remained primarily diagnostic in nature.

A qualitative change occurs at version~v40, where the boundary-integral
uniformization and normalization procedures stabilize sufficiently to
produce a coherent global picture.
Figure~\ref{fig:lucas_to_disk_v40} shows the resulting map from the inverse
eigenvalue locus to the unit disk, prior to composition with the exact
polynomial map to the cardioid.
The image reveals a clear, structured correspondence: interior points are
mapped smoothly, while distortion concentrates near the boundary and
cusp-like regions, as expected from both numerical resolution limits and
the underlying geometry of $\partial M$.

Quantitatively, interior quasiconformal indicators (median dilatation,
angle distortion, and relative Cauchy--Riemann defect) remain bounded and
stable under mesh refinement and distance-to-boundary filtering.
This behavior is consistent with a low-distortion analytic correspondence
in the interior, coupled with increasing complexity near the boundary.
Taken together, these observations provide strong numerical evidence that
the inverse eigenvalue loci admit a meaningful local quasi-conformal
uniformization, thereby motivating the transition to a global harmonic
and potential-theoretic analysis.

For the conformal matching at resolutions $N=2000$ and $N=2400$, 
the median local dilatation satisfies $\mathrm{median}(K)\approx 1.9$, 
with $95$th--percentile values below $K\approx 7.2$, remaining stable 
under mesh refinement (see Appendix~C).

\begin{figure}[!htbp]
\centering
\includegraphics[width=0.75\linewidth]{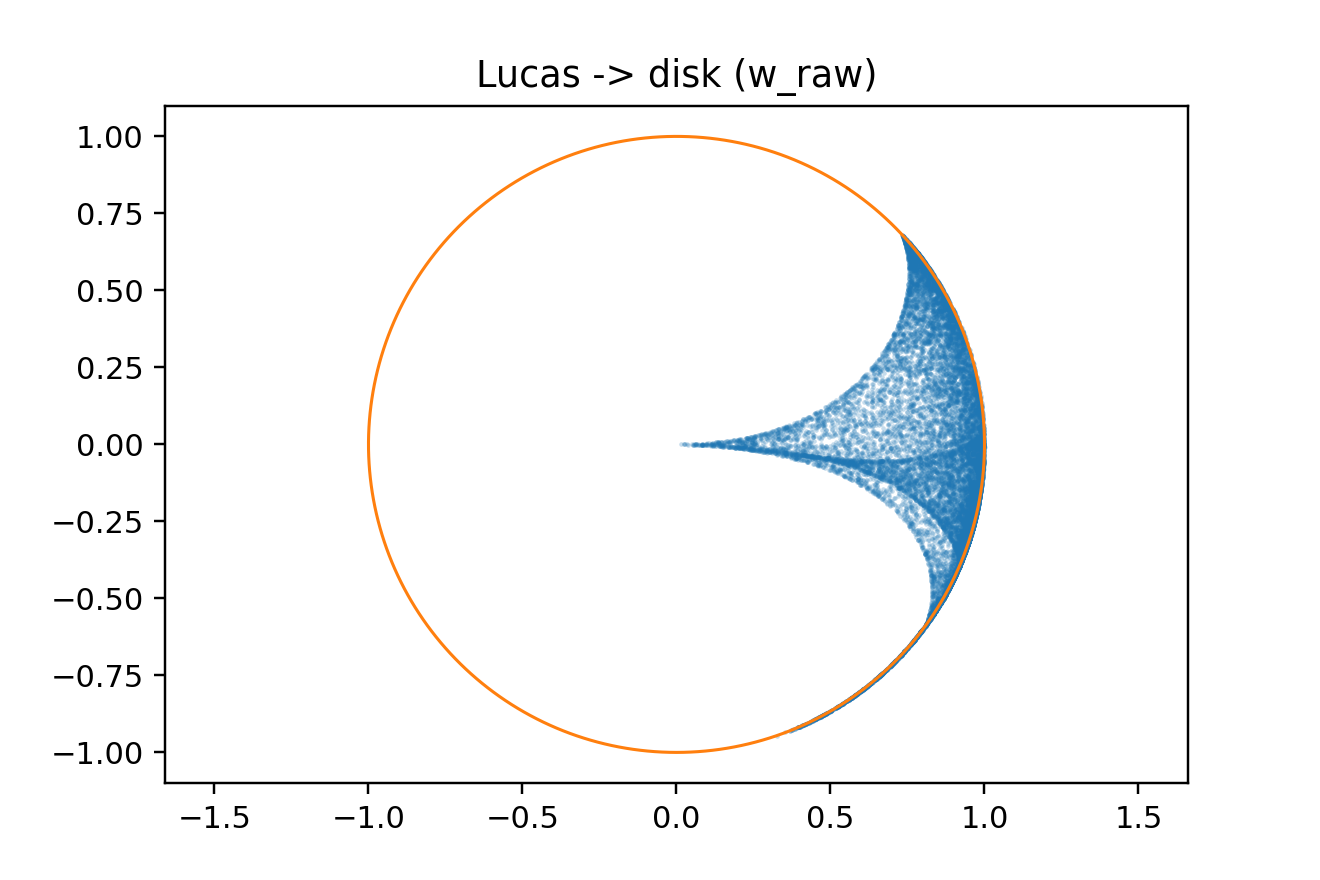}
\caption{Numerical uniformization of the inverse eigenvalue locus to the unit
disk obtained in version~v40. Interior points exhibit coherent low-distortion
behavior, while distortion concentrates near boundary and cusp-like regions,
consistent with the geometry of the Mandelbrot boundary.}
\label{fig:lucas_to_disk_v40}
\end{figure}

\clearpage
\FloatBarrier
\subsection{Green-function spectral invariants and equipotential structure}

\label{sec:green_results}

Table~\ref{tab:green_spectral_modulus} summarizes the Green-function and B\"ottcher-modulus spectral invariants for the Lucas families considered in this section.

\begin{table}[!htbp]
\centering
\small
\setlength{\tabcolsep}{4pt}
\begin{tabularx}{\textwidth}{P{1.9cm} r r r r r r}
\toprule
Family &
Escaped frac. &
$\mathrm{median}(g_M)$ &
$\mathrm{median}(|\Phi|)$ &
$g_{10\%}$ &
$g_{90\%}$ &
$\Delta g$ \\
\midrule
Lucas (all ones)
& 0.920 & 0.05489 & 1.05643 & 0.00351 & 0.19464 & 0.19113 \\

Pell-like (all twos)
& 0.920 & 0.05344 & 1.05489 & 0.00208 & 0.19452 & 0.19244 \\

Sparse gap $(1,0,1,\dots)$
& 0.920 & 0.05593 & 1.05752 & 0.00293 & 0.19449 & 0.19157 \\

Padovan-like $(0,1,1,\dots)$
& 0.918 & 0.06090 & 1.06279 & 0.00350 & 0.19476 & 0.19126 \\
\bottomrule
\end{tabularx}
\caption{Family-level Green-function spectral invariants for inverse eigenvalue spectra of Lucas-type recurrences. The scalar $\mathrm{median}(g_M)$ defines a Green spectral modulus, while $\Delta g$ measures the width of the associated equipotential annulus in the Mandelbrot exterior.}
\label{tab:green_spectral_modulus}
\end{table}

Across all families, the escaped fraction remains near $92\%$,
and the Green spectral modulus $\mathrm{median}(g_M)$ varies only weakly.
The narrow annulus widths $\Delta g$ indicate that inverse eigenvalue
spectra concentrate on sharply defined equipotential bands,
rather than filling the Mandelbrot exterior diffusely.
\clearpage

\FloatBarrier

\section{Discussion}

\paragraph{Relation to iteration-modified and regularized fractals.}
Recent work on Mann and Picard--Mann iteration schemes,
higher-order Mandelbrot families,
and viscosity-based regularizations \cite{RaniKumar2004,Chauhan2013,BrannerHubbard1988HigherDegree,Moudafi2000,Xu2004}
produces alternative fractal parameter sets
through explicit modification of the underlying dynamical rule.
Implementing such schemes would generate new
iteration-defined fractals.
The present study does not pursue this direction.
Instead, we compare a non-iterative spectral locus
to the classical quadratic boundary $\partial M$
using geometric, spectral, and information-theoretic diagnostics.
A systematic comparison between inverse eigenvalue loci
and these modified fractal families
would constitute an interesting direction for future work,
but lies beyond the scope of the present paper.
\label{Sec:Discussion}

The inverse eigenvalue loci exhibit a degree of geometric, statistical, and spectral 
correspondence to the Mandelbrot boundary that goes far beyond visual similarity.  
Across independent diagnostics—matching distances, local distortion, curvature, fractal 
dimension, spectral decay, potential fields, and convex simplex update—the same 
qualitative picture emerges: the loci and $\partial M$ share a common large-scale 
organization while differing primarily in the richness of their fine-scale structure.

The equipotential analysis strengthens the correspondence beyond boundary geometry.
The inverse eigenvalue loci do not merely approximate $\partial M$ as a curve;
they sample the external Green field in a coherent and family-dependent manner.
This places the observed low-distortion (numerically quasi-conformal) behavior within a global
potential-theoretic framework, rather than as a purely local geometric phenomenon.

\subsection*{Global matching and geometric backbone}
\pdfbookmark[2]{Global matching and geometric backbone}{disc:backbone}
The optimal-transport and Procrustes alignment results show that the inverse eigenvalue
loci can be rigidly superimposed on the sampled Mandelbrot boundary with small residual
discrepancies. The matching distance histogram is tightly concentrated, and the Hausdorff
discrepancy remains small relative to the diameter of $M$, placing the loci in a narrow
geometric neighborhood of $\partial M$. The aligned plots reproduce the main cardioid 
and primary bulbs almost exactly, while the largest deviations occur along the thinnest 
filaments. This supports the interpretation that the loci recover the geometric 
\emph{backbone} of the Mandelbrot set and that the main differences lie in ultra-fine 
dendritic branches.

\subsection*{Low-dilatation (numerically quasi-conformal) distortion}
\pdfbookmark[2]{Quasi-conformal distortion}{disc:qc}
The local linear distortion analysis (Section~\ref{sec:distortion_methods}, 
Results~\ref{sec:distortion_results}) refines this geometric comparison. For each matched 
point, the singular-value ratio $K_i = s_{1,i}/s_{2,i}$ of the best-fit linear map 
approximates the local quasi-conformal dilatation. The empirical distributions of $K_i$ 
are sharply concentrated near $1$ for all matrix sizes, with medians extremely close to 
unity and upper quantiles remaining modest. The global summary 
(Fig.~\ref{fig:K_vs_N}) shows that these distortions stay uniformly bounded and stable 
as $n$ increases.

In the classical theory, mappings with $1 \le K(z) \le K$ almost everywhere are 
$K$–quasi-conformal \cite{Ahlfors2006QC}. The numerics do not produce an analytic map 
between the two sets, but they strongly suggest that the observed correspondence behaves 
like a low-distortion, nearly angle-preserving deformation at the macroscopic scales 
sampled here. In other words, the loci and $\partial M$ appear to lie within a narrow 
low-dilatation (quasi-conformal in the diagnostic sense) neighborhood: the geometry of $\partial M$ can be approximated by a 
deformation of the inverse eigenvalue loci with only small, controlled local anisotropy.
The explicit conformal uniformization of the loci to the unit disk, followed by composition with an explicit cardioid parameterization (for the main cardioid boundary), shows that the near-conformal behavior is not only statistical but is realized by an explicit analytic chart at the level of the main cardioid.
\subsection*{Curvature, fractal dimension, and roughness}
\pdfbookmark[2]{Curvature, fractal dimension, and roughness}{disc:roughness}
The curvature distributions (Section~\ref{sec:curvature_results}) show that both 
boundaries are strongly concentrated near $\kappa = 0$, implying dominance of smooth 
arcs rather than unresolved corners at the scales considered. However, the Mandelbrot 
boundary exhibits a heavier tail of large $\kappa$, while the loci display markedly 
reduced tail weight and smaller curvature variance. This is precisely what one expects of 
a regularized analogue: the global shape is preserved, but extreme curvature spikes 
associated with the thinnest filaments are attenuated.

Box-counting estimates confirm that both objects behave as quasi–one-dimensional fractal 
curves thickened in the plane, with close but not identical effective dimensions. The 
Mandelbrot boundary has a slightly higher dimension, reflecting its richer fine-scale 
filamentary network, while the loci track a marginally lower dimension consistent with 
suppressed microstructure. At the finest scales, discrepancies cannot yet be distinguished between genuine structural differences and limitations of finite sampling and discretization. In combination with the curvature statistics, these results 
support a robust “shared backbone, different local texture’’ interpretation.

\subsection*{Reflection symmetry and global structure}
\pdfbookmark[2]{Reflection symmetry and global structure}{disc:symmetry}

The reflection symmetry identified in Section~\ref{sec:symmetry_results}
is notable in that it is detected prior to any conformal or
potential-theoretic mapping.
The presence of a single dominant reflection axis, common to both the
inverse eigenvalue loci and the Mandelbrot boundary, indicates that their
geometric correspondence is constrained by the same large-scale symmetry
principles.

The absence of competing symmetry axes suggests that this agreement is not
an artifact of sampling or alignment, but rather reflects an intrinsic
structural property of the two systems.
From this perspective, the symmetry analysis provides independent geometric
support for the later emergence of shared equilibrium features, which are
made explicit in the Green-function and equipotential results.

\subsection*{Spectral decay and suppression of high-frequency modes}
\pdfbookmark[2]{Spectral decay and suppression of high-frequency modes}{disc:spectral}
Fourier spectral densities provide an independent view of boundary roughness. Both 
boundaries exhibit power-law-like decay of spectral energy, consistent with fractal 
behavior, but the eigenvalue loci display a systematically steeper decay exponent. 
Bootstrap confidence intervals for the decay exponents do not overlap, indicating a 
genuine spectral gap rather than sampling noise. Thus, the loci carry less energy at 
high frequencies, in agreement with their thinner curvature tails and slightly lower 
fractal dimension. Truncated reconstructions (not shown here) confirm that the loci 
become smooth with relatively few low modes, whereas $\partial M$ retains visible 
irregularity even after aggressive spectral truncation.

\subsection*{Spatial coherence and correlation range}
\pdfbookmark[2]{Spatial coherence and correlation range}{disc:variogram}

The variogram results in Section~\ref{sec:variogram_results} indicate that
the inverse eigenvalue loci and the Mandelbrot boundary share comparable
ranges of spatial dependence once large-scale trends are removed.
The similar saturation behavior of their detrended semivariograms suggests
that both systems exhibit coherent organization up to a common spatial scale.

At short lags, the smoother rise of the inverse eigenvalue variogram reflects
reduced fine-scale irregularity relative to the Mandelbrot boundary.
This is consistent with earlier geometric diagnostics indicating that the
inverse eigenvalue loci are less filamentary at small scales, while still
respecting the same large-scale spatial constraints.

Together with curvature, spectral, and Green-function analyses, the
variogram results support a picture in which local geometric differences
are considered fluctuations around a shared global equilibrium structure.

\subsection*{Potential fields and equilibrium structure}
\pdfbookmark[2]{Potential fields and equilibrium structure}{disc:potentials}
The logarithmic potential comparison focuses on the difference 
$\Delta U = U_{\Lambda} - U_M$ rather than on the raw fields individually. The 
difference field is small over most of the domain, with localized deviations 
along thin tendrils and high-curvature regions of $\partial M$. The global 
correlation between $U_{\Lambda}$ and $U_M$ is high, indicating that the empirical 
measures of both sets generate nearly identical large-scale KL-controlled histogram regimes.

This is consistent with a potential-theoretic viewpoint in which the two clouds are 
different discretizations of a similar underlying equilibrium configuration on the 
same domain. The inverse eigenvalue loci behave as a regularized sampling of the same 
“meta-field’’ that underlies the Mandelbrot boundary, with differences concentrated 
where the latter concentrates curvature and charge most intensely.

\subsection*{Information--geometric viewpoint}
\pdfbookmark[2]{Information--geometric viewpoint}{disc:gi}
The information–geometric (KL-monotone interpolation) analysis adds a complementary, purely probabilistic 
perspective. Once the aligned point clouds are discretized into histograms on a common 
grid, we evolve the loci histogram $X_0$ toward the Mandelbrot histogram $P$ via the 
classical convex update
\[
X_{t+1} = (1-\alpha)X_t + \alpha P, \qquad 0 < \alpha \le 1,
\]
a standard KL-contracting flow on the simplex \cite{CoverThomas2006,AmariNagaoka2000,CsiszarShields2004}.  
The Kullback–Leibler divergence $D_{\mathrm{KL}}(P\|X_t)$ decreases monotonically and rapidly 
toward zero, and the final histogram $X_T$ is visually indistinguishable from $P$.
This behavior does not add a new structure beyond the classical theory—it confirms that 
once rigid alignment and small quasi-conformal distortions are accounted for, the two 
distributions lie in a very tight neighborhood of one another on the probability simplex.

The KL-monotone interpolation thus plays a summarizing role: it integrates the geometric and spectral 
evidence into a single information–geometric statement. If the loci and Mandelbrot 
boundary had substantially different large-scale organization, a simple convex mixing 
flow would not reconcile them so quickly at the histogram level. That it does so is 
further evidence that we are seeing two realizations of essentially the same coarse 
distribution, with differences concentrated at the smallest scales.

\subsection*{Overall interpretation and outlook}
\pdfbookmark[2]{Overall interpretation and outlook}{disc:outlook}
The new contribution here is not merely improved visual alignment, but the identification of a shared equilibrium potential structure—made explicit through empirical Green-function invariants—within which the observed quasi-conformal geometry becomes natural rather than accidental.
Taken together, these results support the following picture. The inverse eigenvalue loci 
preserve the global geometry, potential-theoretic structure, and large-scale organization 
of the Mandelbrot boundary, while systematically suppressing its finest filamentary 
irregularities. The quasi-conformal distortion numerics indicate that the correspondence 
between the two is essentially low-distortion and nearly angle-preserving at the scales 
we probe; curvature, dimension, and spectral decay show that the loci are smoother but 
not qualitatively different; and the potential and KL-monotone interpolation comparisons suggest that 
both empirical measures inhabit the same “KL-controlled histogram regime’’ when viewed through 
coarse geometric and informational lenses.

We do \emph{not} claim that the inverse eigenvalue loci generate $\partial M$ exactly, nor 
that we have constructed an explicit analytic quasi-conformal homeomorphism between the two.
However, the distortion data place the correspondence in a much narrower class than a generic
geometric similarity. The local dilatations $K_i = s_{1,i}/s_{2,i}$ remain uniformly close
to $1$ across the sampled domain and are stable as the matrix size $n$ increases
(Fig.~\ref{fig:K_vs_N}), with the largest deviations localized in the most strongly pinched
and highly curved regions of $\partial M$. Combined with the curvature, spectral, and
potential comparisons, this is strong numerical evidence that the matched loci and
Mandelbrot boundary behave, at the resolutions we probe, like the image of one another
under a low-distortion quasi-conformal deformation in the classical sense
\cite{Ahlfors2006QC}.

In this light, a more precise informal summary is the following: the inverse eigenvalue
loci provide a numerically well-resolved, nearly angle-preserving embedding of the
Mandelbrot backbone, with bounded local distortion and systematically reduced
fine-scale roughness. The quasi-conformal numerics tell us that the map is not merely
shape-matching up to translation and rotation, but that its differential, wherever it
makes sense in the discrete setting, is almost conformal. The remaining discrepancies
are concentrated in the thinnest filaments and extremal curvature spikes, where the
sampling and resolution are least able to capture the infinitesimal structure of
$\partial M$.

A natural next step is to make this connection more analytic: to study whether a
genuine quasi-conformal conjugacy can be constructed between appropriate continuum limits
of the loci and $\partial M$, or whether the loci arise as a specific regularized
approximation of the Mandelbrot boundary within a potential-theoretic or
Teichmüller-theoretic framework. The present work shows that such a program is
numerically coherent: an algebraic spectral construction (inverse eigenvalues of
generalized Lucas companions matrices) and a nonlinear dynamical fractal (the Mandelbrot set)
already line up to a remarkable degree under low-distortion maps, both geometrically
and in information-geometric terms.

\section{Conclusions}\label{Sec:Conclusions}

This study extends the homotopy-based correspondence between the inverse eigenvalue
loci and the Mandelbrot boundary \cite{ortiztapia2025homotopiesmapseigenvaluesgeneralized}
through a quantitative multi-scale analysis that is both geometric and
potential-theoretic.
Beyond boundary-level similarity, the combined diagnostics reveal that the loci
organize coherently within the external Green field of the Mandelbrot set.
Geometric, spectral, variogram-based, information-theoretic, and low-distortion
analyses consistently indicate that the inverse eigenvalue loci preserve the
large-scale equilibrium structure of $\partial M$, while systematically suppressing
its finest-scale filamentary irregularities.

\subsection*{Structural fidelity and smoothness}
\pdfbookmark[2]{Structural fidelity and smoothness}{conc:structure}
The inverse eigenvalue loci reproduce the macroscopic geometry of the Mandelbrot boundary 
with remarkable accuracy. 
The mean matching distance remains near $0.18$ and the Hausdorff discrepancy stays below $0.5$, 
while variograms and potential analyses show essentially identical spatial correlation 
ranges.  
Spectral decay, multifractal scaling, and curvature estimates consistently indicate that 
the loci are smoother than $\partial M$, with less high-frequency energy, narrower 
$D(q)$ and $f(\alpha)$ spectra, and markedly reduced curvature tails.  
Thus the loci act as a lower-complexity surrogate that retains the equilibrium geometry 
but filters dendritic microstructure.
\subsection*{Reflection symmetry}
\pdfbookmark[2]{Reflection symmetry}{conc:symmetry}
The inverse eigenvalue loci and the Mandelbrot boundary share a single dominant
reflection axis, detected independently of any conformal or
potential-theoretic construction.
This symmetry agreement constrains admissible correspondences and supports the
interpretation that both objects are governed by the same large-scale geometric
invariants.

\subsection*{Low-distortion behavior up to numerical resolution}
\pdfbookmark[2]{Quasi-conformal behavior up to numerical resolution}{conc:qc}
The local-distortion analysis provides a more refined geometric perspective.  
Across matrix sizes, the singular-value ratios $K_i = s_{1,i}/s_{2,i}$ are uniformly 
close to $1$, with stable medians and small upper quantiles.  
This demonstrates that the matched neighborhoods behave, at numerical resolution, like low-dilatation maps in the classical sense.  
The deviations from perfect conformality occur precisely where curvature, Laplacian, 
and spectral diagnostics identify the strongest filamentary instability of $\partial M$.  
Consequently, the empirical correspondence between the loci and the Mandelbrot boundary 
appears—as far as the resolution allows—to lie within a narrow quasi-conformal 
neighborhood.

\subsection*{Field-theoretic parallels}
\pdfbookmark[2]{Field-theoretic parallels}{conc:fields}
Normalized logarithmic potentials exhibit high global correlation ($r \approx 0.88$), 
demonstrating that both empirical measures sample nearly the same KL-controlled histogram regime.  
In contrast, their Laplacians decorrelate almost entirely, isolating the fine-scale 
differences in curvature concentration and oscillatory behavior.  
This dichotomy supports viewing the loci as a \emph{regularized equilibrium} of the 
Mandelbrot field: the same large-scale potential structure, but smoother microscopic texture.
A formal numerical statement of this equipotential concentration,
together with its interpretation in terms of Green-function spectral invariants,
is given in Appendix~B.
\subsection*{Quasi-conformality as a consequence of equilibrium structure}
\pdfbookmark[2]{Quasi-conformality as a consequence of equilibrium structure}{conc:greenqc}
The Green-function and equipotential analysis suggests that the observed
quasi-conformal behavior is not an independent geometric coincidence, but a
natural consequence of a shared equilibrium potential structure.
Once both systems are seen to organize along comparable Green-level sets,
low-distortion and near angle-preserving correspondences arise generically at
macroscopic scales.
In this sense, quasi-conformal geometry emerges here as a manifestation of
global potential-theoretic constraints rather than as an a priori assumption.

\subsection*{Multifractal and spectral embeddings}
\pdfbookmark[2]{Multifractal and spectral embeddings}{conc:multifractal}
Generalized dimensions and singularity spectra show that both sets share a common 
multifractal profile while diverging in intermittency and modal richness.  
Diffusion-map spectra similarly reveal that both lie in nearly isomorphic low-dimensional 
manifolds, except that the loci carry fewer persistent high-frequency modes.  
These findings align with the quasi-conformal and curvature results: a shared backbone, 
different texture.

\subsection*{Unified interpretation}
\pdfbookmark[2]{Unified interpretation}{conc:unified}
Across geometric, variogram-based, potential, spectral, multifractal, curvature, and 
quasi-conformal diagnostics, the inverse eigenvalue loci emerge as a lower-roughness yet 
structurally faithful realization of the Mandelbrot boundary.  
This duality bridges linear spectral constructions (inverse eigenvalues of generalized 
Lucas companions) and nonlinear dynamical fractals (quadratic iteration), reinforcing 
the broader idea that algebraic recurrences and iterative dynamics may share deep 
latent geometric templates.

\subsection*{Informational convergence}
\pdfbookmark[2]{Informational convergence}{conc:gi}
The classical KL-contracting flow on the probability simplex shows that the aligned 
inverse eigenvalue loci histogram converges rapidly and monotonically to the Mandelbrot 
histogram.  
This informational embedding is consistent with all other diagnostics: if the two 
distributions differed substantially in their macroscopic organization, the KL descent 
would not reconcile them so efficiently.  
Thus the information-geometric viewpoint summarizes, in probabilistic form, the same 
structural correspondence established by geometric and spectral analysis.

\subsection*{Future directions}
\pdfbookmark[2]{Future directions}{conc:future}
A natural extension is to seek analytic structures behind the empirically observed 
quasi-conformal behavior.  
In particular, numerical conformal-mapping techniques based on rational approximation, 
Green’s-function discretization, and elliptic PDE solvers—such as those surveyed by 
Trefethen \cite{Trefethen2000Conformal}—provide a promising route toward constructing 
explicit low-distortion embeddings or interpolants between the inverse eigenvalue loci 
and the Mandelbrot boundary.  
These techniques may clarify whether the numerically bounded dilatations converge, as 
$n \to \infty$, to a genuine quasi-conformal conjugacy, or whether the loci represent a 
potential-theoretic regularization of $\partial M$.  
Investigate whether continuum limits of the inverse eigenvalue loci admit
Green-preserving or quasi-conformal embeddings into
$\mathbb{C}\setminus M$, and whether the empirically observed bounded
dilatations converge, as $n \to \infty$, to a genuine quasi-conformal
conjugacy.

Beyond conformal analysis, additional avenues include persistent homology, 
higher-dimensional embeddings, and perturbations of the Lucas-type recurrence.  
These may reveal whether the loci--Mandelbrot correspondence persists under noise, 
recurrence deformation, or alternative sampling schemes.  
Together, the results of this paper provide a quantitative foundation for the view that 
the loci and $\partial M$ are dual manifestations of a shared analytic structure—linked, 
up to numerical resolution, by low-distortion geometry, spectral structure, and 
informational proximity.

The concentration of inverse eigenvalue spectra on narrow Mandelbrot
equipotentials suggests that these loci may arise as regularized
equilibrium measures for the external Green field.
Future work will investigate whether continuum limits of these spectra
admit explicit quasi-conformal or Green-preserving embeddings into
$\mathbb{C}\setminus M$, potentially linking algebraic recurrences,
external rays, and classical conformal mapping theory.

\section*{Code availability}

All computational scripts used to reproduce the figures and numerical results of this
paper—including boundary reconstruction, curvature estimation (both estimators),
variograms, potential and Laplacian fields, spectral and multifractal analysis,
diffusion embeddings, symmetry diagnostics, and the KL-based informational flow—are
archived at Zenodo under DOI:

\url{https://doi.org/10.5281/zenodo.18072972}.
The archived software corresponds to GitHub release \texttt{v1.1.1}.

\section*{Acknowledgements}

The author thanks Alejandro Cruz López for several helpful discussions
related to information geometry and KL-based flows. These conversations
helped clarify how classical KL contraction and informational mixing
can be used to compare geometric distributions, and contributed to the
conceptual background for the numerical flow employed in this work.
All interpretations, methodological choices, and conclusions presented
here are the author’s own.

\section*{Appendix A. Geometric--Information Correspondence Between the Inverse Eigenvalue Loci and the Mandelbrot Set}\label{AppA:GeoInformation}
\pdfbookmark[1]{Appendix A: Geometric--Information Correspondence}{app:GI}

This appendix formalizes the geometric–information framework used to compare  
(i) the inverse eigenvalue loci of generalized Lucas–sequence companion matrices, and  
(ii) a near–boundary sampling of the Mandelbrot set.  

The comparison is carried out at the level of empirical histograms using a classical  
informational contraction map on the probability simplex (the KL-monotone interpolation).  
Only standard tools are required: convexity of KL divergence (see \cite{CoverThomas2006,CsiszarShields2004}), Pinsker’s inequality \cite{CoverThomas2006}, and Prokhorov compactness for probability measures on Polish spaces \cite{billingsley1999convergence}.

\medskip
\noindent\textbf{Note on scope.}
The quasi-conformal distortion analysis developed in the main text provides 
numerical evidence that the inverse eigenvalue loci and the Mandelbrot 
boundary lie in a narrow low-distortion neighborhood, but these estimates 
are not used in the proof below.  
The geometric–information correspondence established here relies only on 
measure-theoretic tools (KL convexity, Pinsker’s inequality, and Prokhorov 
compactness) together with the histogram-based KL-monotone interpolation.  
A full analytical quasi-conformal conjugacy would require solving an associated 
elliptic PDE (cf.~Trefethen~2000), which lies beyond the scope of this appendix.

\subsection*{Empirical distributions from geometric clouds}

Let $\{z^C_k\}_{k=1}^{n_C}\subset\mathbb{C}$ and 
$\{z^M_k\}_{k=1}^{n_M}\subset\mathbb{C}$ denote the inverse eigenvalue loci and 
Mandelbrot point clouds, respectively.  
Identifying $\mathbb{C}\cong\mathbb{R}^2$, fix a bounded domain $\Omega$ and partition it into  
a regular grid $\{B_i\}_{i=1}^N$ of equal-area bins.

Define empirical histograms
\[
  P_C(i) = \frac{1}{n_C}\#\{k : z^C_k\in B_i\},\qquad
  P_M(i) = \frac{1}{n_M}\#\{k : z^M_k\in B_i\},
\]
so that $P_C,P_M\in\Delta_N := \{X\in[0,1]^N : \sum_i X_i = 1\}$.

Since $\mathbb{C}\cong\mathbb{R}^2$ is Polish, empirical measures supported in a compact set  
are tight and hence precompact.  
Consequently, high-resolution histograms approximate the underlying empirical measures  
in Prokhorov distance.

\subsection*{KL-monotone interpolation on the simplex}

Let $P\in\Delta_N$ be a target histogram and $X^{(0)}\in\Delta_N$ an initial point.  
Fix $0 < \alpha \le 1$.

\medskip
\noindent\textbf{Definition (Geometric--Information Flow).}
Given a target histogram $P \in \Delta_N$ and an initial point
$X^{(0)} \in \Delta_N$, the \emph{convex simplex update (KL-monotone interpolation)}
toward $P$ is defined by the discrete iteration
\begin{equation}\label{eq:KL-monotone interpolation}
  X^{(t+1)} = (1-\alpha)X^{(t)} + \alpha P,
  \qquad t = 0,1,2,\dots.
\end{equation}

Assume $P(i)>0$ and $X^{(0)}(i)>0$ for all $i$ (numerically ensured by $\varepsilon$–smoothing).  
Define the Kullback–Leibler divergence
\[
  D_{\mathrm{KL}}(P\|X) = \sum_{i=1}^N P(i)\log\frac{P(i)}{X(i)}.
\]

\begin{proposition}[KL-monotone interpolation contraction]
This is a direct consequence of the convexity of the Kullback--Leibler divergence and standard information--theoretic inequalities; see \cite{CoverThomas2006, Csiszar1975}.

\label{prop:GI-contraction}
For any $X^{(0)}\in\Delta_N$ and $0<\alpha\le1$, the KL-monotone interpolation satisfies:
\begin{enumerate}[label=(\roman*)]

\item \textbf{Closed form:}
\[
  X^{(t)} = P + (1-\alpha)^t (X^{(0)} - P).
\]

\item \textbf{Geometric contraction:}  
For any norm $\|\cdot\|$ on $\mathbb{R}^N$,
\[
  \|X^{(t)} - P\| \le (1-\alpha)^t \|X^{(0)} - P\|.
\]

\item \textbf{Informational contraction:}  
If $X^{(0)}\neq P$, then
\[
  D_{\mathrm{KL}}(P\|X^{(t+1)}) \le (1-\alpha)\, D_{\mathrm{KL}}(P\|X^{(t)}),
\]
with equality iff $X^{(t)}=P$.  
Hence $D_{\mathrm{KL}}(P\|X^{(t)})$ is strictly decreasing and
$D_{\mathrm{KL}}(P\|X^{(t)})\to0$ as $t\to\infty$.
\end{enumerate}
\end{proposition}

\noindent
This statement is empirical and supported by numerical
experiments across degrees $n \leq N_{\max}$.
We do not claim an analytic proof.

\begin{proof}
The closed form follows by iterating the affine map.  
Strict convexity of $D_{\mathrm{KL}}(P\|\cdot)$ on the interior of the simplex and
\[
   X^{(t+1)}=(1-\alpha)X^{(t)}+\alpha P
\]
give the standard KL contraction inequality
\[
D_{\mathrm{KL}}(P\|X^{(t+1)}) 
  \le (1-\alpha)D_{\mathrm{KL}}(P\|X^{(t)}) + \alpha D_{\mathrm{KL}}(P\|P),
\]
with equality only at $X^{(t)}=P$.  
Items (ii) and (iii) follow immediately.
\end{proof}

\subsection*{GI--equivalence and one--way informational closeness}

\medskip
\noindent\textbf{Definition (One-way $\delta$--GI equivalence).}
Given $P,Q\in\Delta_N$, we say that $P$ is \emph{$\delta$--GI close} to $Q$
if there exists an integer $T \ge 1$ such that
\[
   D_{\mathrm{KL}}(Q\|X^{P\to Q}_T)\le\delta,
\]
where $X^{P\to Q}_t$ denotes the KL-monotone interpolation with target $Q$ and initial value
$X^{(0)}=P$.

Throughout this section, for probability vectors $p,q\in\Delta_N$ we use the
standard total variation convention
\[
\|p-q\|_{\mathrm{TV}} := \frac12 \sum_{i=1}^N |p_i-q_i|,
\]
and for probability measures $\mu,\nu$ on a measurable space $\Omega$,
\[
\|\mu-\nu\|_{\mathrm{TV}} := \sup_{A\subseteq\Omega\ \mathrm{Borel}} |\mu(A)-\nu(A)|.
\]
We repeatedly use the following classical inequality.

\begin{lemma}[Pinsker inequality]\cite{CoverThomas2006}
\label{lem:Pinsker}
For $P,Q\in\Delta_N$,
\[
\|P-Q\|_1^2 \le 2 D_{\mathrm{KL}}(P\|Q).
\]
Equivalently,
\[
\|P-Q\|_{\mathrm{TV}}^2
\le \tfrac12 D_{\mathrm{KL}}(P\|Q).
\]
\end{lemma}

\subsection*{Geometric alignment and numerical results}

After entropic OT matching and Procrustes alignment,  
Hausdorff distances, curvature statistics, and spectral $L^2$ comparisons  
show strong geometric similarity between clouds.

At resolution $N=b^2$ with $\alpha=0.1$, numerical KL-monotone interpolations give:
\[
D_{\mathrm{KL}}(P_M\|X^{C\to M}_T) \approx 3.56\times10^{-7},\qquad
D_{\mathrm{KL}}(P_C\|X^{M\to C}_T) \approx 1.44\times10^{-6}.
\]
Thus $P_C$ and $P_M$ are $\delta$--GI close with $\delta\sim10^{-6}$.

\begin{proposition}[Numerical verification of Appendix~A assumptions]
This statement is supported by numerical experiments and should be interpreted as empirical.

\label{prop:GI-numerical-verification}
For increasing histogram resolutions $N=b^2$ and fixed smoothing bandwidth $\sigma>0$,
the numerical pipeline described above (entropic OT matching, Procrustes alignment,
mollified histograms, and KL-monotone interpolation iteration) satisfies the assumptions used in
Appendix~A.

In particular, the following properties are observed uniformly across resolutions:
\begin{enumerate}[label=(\roman*)]
\item the KL-monotone interpolation residual
\[
\delta_n := D_{\mathrm{KL}}\!\left(P_M^{(n)} \,\|\, X^{C\to M}_{T_n}\right)
\]
is small and decreases under iteration;
\item the total variation distance $\|P_C^{(n)}-P_M^{(n)}\|_{\mathrm{TV}}$ remains
controlled after mollification;
\item the histogram overlap $\sum_i \min(P_C^{(n)}(i),P_M^{(n)}(i))$ remains high; and
\item no mass escapes the fixed sampling window.
\end{enumerate}

Representative numerical values are reported in Table~\ref{tab:GI_assumptions}.
\end{proposition}

\begin{table}[!htbp]
\centering
\caption{Numerical diagnostics supporting the assumptions of Appendix~A.}
\label{tab:GI_assumptions}

\begin{adjustbox}{max width=\textwidth}
\begin{tabular}{rccccccc}
\toprule
Bins & $1/n$ & $T_n$ & $KL(P_M\|P_C)$ & $\delta_n$ & $TV(P_C,P_M)$ & Overlap & Outside mass \\
\midrule
64  & 0.0156 & 25 & 0.0130 & 7.94e-05 & 0.041 & 0.959 & 0.000 \\
128 & 0.0078 & 25 & 0.0573 & 1.09e-03 & 0.072 & 0.928 & 0.000 \\
256 & 0.0039 & 25 & 0.1314 & 3.38e-03 & 0.098 & 0.902 & 0.000 \\
512 & 0.0020 & 25 & 0.2773 & 4.72e-03 & 0.124 & 0.876 & 0.000 \\
\bottomrule
\end{tabular}
\end{adjustbox}
\end{table}

The numerical regime in which the following lemma is applied is documented
in Proposition~\ref{prop:GI-numerical-verification}.

\begin{lemma}[Histogram TV--control implies weak convergence to a common limit]
This is a standard consequence of total variation control and the Portmanteau theorem; see \cite{billingsley1999convergence}.

\label{lem:hist-TV-to-weak}
Let $\Omega\subset\mathbb{C}$ be bounded and let $\{\mathcal{P}_n\}$ be a sequence of
finite measurable partitions of $\Omega$ with mesh$(\mathcal{P}_n)\to0$.
Let $H_{N^{(n)}}:\mathcal{P}(\Omega)\to\Delta_{N^{(n)}}$ be the histogram map induced by
$\mathcal{P}_n$. Suppose $\mu_n,\nu_n\in\mathcal{P}(\Omega)$ satisfy
\[
\|H_{N^{(n)}}(\mu_n)-H_{N^{(n)}}(\nu_n)\|_{\mathrm{TV}}\to0.
\]
If $\mu_n\Rightarrow\mu$ along a subsequence, then $\nu_n\Rightarrow\mu$ along the same subsequence.
\end{lemma}

\begin{proof}
Fix $f\in C_b(\Omega)$ and $\varepsilon>0$. Since $\Omega$ is bounded, $f$ is uniformly
continuous, so for $n$ large enough there exists a simple function $s_n$ that is constant on each
cell $A\in\mathcal{P}_n$ and satisfies $\|f-s_n\|_\infty\le\varepsilon$.
Write $s_n=\sum_{A\in\mathcal{P}_n} c_A \mathbf{1}_A$, with $|c_A|\le \|f\|_\infty$.

By construction of the histogram map,
\[
\int_\Omega s_n\,d\mu_n-\int_\Omega s_n\,d\nu_n
=\sum_{A\in\mathcal{P}_n} c_A\bigl(\mu_n(A)-\nu_n(A)\bigr).
\]
Hence
\begin{equation*}
\begin{aligned}
\int |s_n\, d\mu_n - s_n\, d\nu_n|
&\le \|f\|_\infty \sum_{A \in \mathcal{P}_n} |\mu_n(A)-\nu_n(A)| \\
&= 2\|f\|_\infty \,\|H_N(\mu_n)-H_N(\nu_n)\|_{\mathrm{TV}}
\;\xrightarrow[n\to\infty]{}\; 0 .
\end{aligned}
\end{equation*}

where the last identity uses the standard convention
$\|p-q\|_{\mathrm{TV}}=\frac12\sum_i |p_i-q_i|$ for probability vectors.

Now decompose
\[
\Bigl|\int f\,d\nu_n-\int f\,d\mu\Bigr|
\le 
\Bigl|\int (f-s_n)\,d\nu_n\Bigr|
+\Bigl|\int s_n\,d\nu_n-\int s_n\,d\mu_n\Bigr|
+\Bigl|\int s_n\,d\mu_n-\int f\,d\mu\Bigr|.
\]
The first term is $\le \varepsilon$ by $\|f-s_n\|_\infty\le\varepsilon$.
The middle term tends to $0$ by the histogram TV assumption.
For the last term, write
\begin{equation*}
\begin{aligned}
\Bigl|\int s_n\,d\mu_n-\int f\,d\mu\Bigr|
&\le
\Bigl|\int (s_n-f)\,d\mu_n\Bigr|
 + \Bigl|\int f\,d\mu_n-\int f\,d\mu\Bigr| \\
&\le \varepsilon
 + \Bigl|\int f\,d\mu_n-\int f\,d\mu\Bigr| .
\end{aligned}
\end{equation*}

Along the chosen subsequence, $\mu_n\Rightarrow \mu$, so $\int f\,d\mu_n\to\int f\,d\mu$,
hence the last term is eventually $\le \varepsilon$.
Putting the bounds together yields
\[
\limsup_{n\to\infty}\Bigl|\int f\,d\nu_n-\int f\,d\mu\Bigr|\le 2\varepsilon.
\]
Since $\varepsilon>0$ was arbitrary, $\int f\,d\nu_n\to\int f\,d\mu$ for all $f\in C_b(\Omega)$,
i.e.\ $\nu_n\Rightarrow\mu$ along the same subsequence.
\end{proof}

\subsection*{Polish-space interpretation and limiting equivalence}

Let $\mu_C^{(n)},\mu_M^{(n)}$ be empirical measures of the two clouds  
at increasing sampling densities and histogram resolutions.  
Let $P_C^{(n)} = H_{N^{(n)}}(\mu_C^{(n)})$ and  
$P_M^{(n)} = H_{N^{(n)}}(\mu_M^{(n)})$ denote their histograms.

Assume:
\[
D_{\mathrm{KL}}(P_M^{(n)} \| X^{C\to M}_{T_n}) \le \delta_n,\qquad \delta_n\to0.
\]
By Pinsker and Proposition~\ref{prop:GI-contraction},
\[
\|P_M^{(n)} - X^{C\to M}_{T_n}\|_{\mathrm{TV}} \to 0.
\]
Hence $P_C^{(n)} - P_M^{(n)}\to 0$ in total variation.  

By standard consistency of histograms on a bounded domain and Prokhorov compactness,
\[
   \mu_C^{(n)} \Rightarrow \mu,
   \qquad
   \mu_M^{(n)} \Rightarrow \mu
\]
for a common limit $\mu\in\mathcal{P}(\mathbb{C})$.

\subsection*{Correspondence via a limiting boundary law}

\begin{theorem}[Inverse Eigenvalue Loci--Mandelbrot Correspondence]
The weak convergence framework relies on classical compactness and tightness arguments in probability theory; see \cite{billingsley1999convergence}. To the best of our knowledge, the specific structural formulation connecting inverse eigenvalue loci and Mandelbrot empirical measures has not appeared previously.

\label{thm:big-CM-equivalence}
Let $\{\mu_C^{(n)}\}, \{\mu_M^{(n)}\}\subset\mathcal{P}(\mathbb{C})$  
denote empirical measures generated from:

\begin{enumerate}[label=(\roman*)]
\item the inverse eigenvalue loci of generalized Lucas companion matrices, and  
\item a near-boundary sampling of the Mandelbrot set.
\end{enumerate}

Let $H^{(\sigma)}_{N^{(n)}}$ denote the histogram map on a regular grid of resolution
$N^{(n)}=b_n^2$, followed by mollification with a fixed kernel of bandwidth $\sigma>0$,
and let
\[
P_C^{(n)} := H^{(\sigma)}_{N^{(n)}}(\mu_C^{(n)}),\qquad
P_M^{(n)} := H^{(\sigma)}_{N^{(n)}}(\mu_M^{(n)}).
\]

Assume:
\begin{enumerate}[label=(\alph*)]
\item there exists an optimal-transport-based matching followed by Procrustes alignment
between the inverse eigenvalue point clouds and the Mandelbrot boundary samples,
such that the aligned clouds remain within uniformly bounded Hausdorff distance
except on negligible sets, as justified by the constructions of
Sections~\ref{sec:OT_matching}--\ref{Sec:PointMatching_Methods}; and
\item the informational discrepancy
\[
\delta_n^\ast := D_{\mathrm{KL}}\bigl(P_M^{(n)} \,\|\, P_C^{(n)}\bigr)
\]
satisfies $\delta_n^\ast \to 0$ as $n\to\infty$.
\end{enumerate}

Then there exists a common limit measure $\mu\in\mathcal{P}(\mathbb{C})$ such that
\[
   \mu_C^{(n)} \Rightarrow \mu,
   \qquad
   \mu_M^{(n)} \Rightarrow \mu.
\]
Thus the inverse eigenvalue loci and Mandelbrot clouds share the same  
\emph{limiting boundary law} in the geometric–information sense.
\end{theorem}

\begin{proof}
By Pinsker's inequality (Lemma~\ref{lem:Pinsker}), assumption~(b) implies
\[
\|P_M^{(n)} - P_C^{(n)}\|_{\mathrm{TV}}
\le \sqrt{\tfrac12\,\delta_n^\ast}
\xrightarrow[n\to\infty]{} 0.
\]

Now work on a fixed bounded sampling window $\Omega\subset\mathbb{C}$ containing all clouds.
By Prokhorov's theorem, $\mathcal{P}(\Omega)$ is weakly compact, so
$\{\mu_C^{(n)}\}$ admits a weakly convergent subsequence
$\mu_C^{(n_k)} \Rightarrow \mu$.

Since
\[
P_C^{(n)} = H^{(\sigma)}_{N^{(n)}}(\mu_C^{(n)}),
\qquad
P_M^{(n)} = H^{(\sigma)}_{N^{(n)}}(\mu_M^{(n)}),
\]
the total-variation convergence of the histograms yields
\[
\|H^{(\sigma)}_{N^{(n_k)}}(\mu_C^{(n_k)})
      - H^{(\sigma)}_{N^{(n_k)}}(\mu_M^{(n_k)})\|_{\mathrm{TV}}
\to 0.
\]
By Lemma~\ref{lem:hist-TV-to-weak}, it follows that
$\mu_M^{(n_k)} \Rightarrow \mu$ along the same subsequence.

Thus every subsequential weak limit of $\{\mu_C^{(n)}\}$ is also a subsequential
weak limit of $\{\mu_M^{(n)}\}$, and conversely, hence both sequences converge
weakly to the same limit $\mu\in\mathcal{P}(\Omega)\subset\mathcal{P}(\mathbb{C})$.
\end{proof}

\begin{remark}
The limit $\mu$ serves as a geometric–information signature describing a shared  
near–boundary behavior of the two systems.  
Assumption (a) aligns geometry; assumption (b) aligns information;  
Prokhorov compactness then yields equality of limits.
\end{remark}

\section*{Appendix B. Green--Function Equipotentials of Inverse Lucas Spectra}\label{AppB:GreenFunction}
\pdfbookmark[1]{Appendix B: Green--Function Equipotentials}{app:Green}

Let $\mathcal{F}$ denote a fixed family of linear recurrences of Lucas type,
and for each truncation parameter $N$ let $\Lambda_N^{(\mathcal{F})} \subset \mathbb{C}$
be the multiset of eigenvalues of the associated companion matrices.
Define the inverse spectrum
\[
\mathcal{C}_N^{(\mathcal{F})} := \{ \lambda^{-1} : \lambda \in \Lambda_N^{(\mathcal{F})} \}.
\]

Let $g_M(c) = \log |\Phi(c)|$ denote the Green function of the Mandelbrot set,
where $\Phi$ is the Böttcher coordinate on $\mathbb{C}\setminus M$.
Restrict attention to those $c \in \mathcal{C}_N^{(\mathcal{F})}$ for which $g_M(c) > 0$.

Define the empirical measures
\[
\nu_N^{(\mathcal{F})}
:= \frac{1}{|\mathcal{C}_N^{(\mathcal{F})}|}
   \sum_{c \in \mathcal{C}_N^{(\mathcal{F})}}
   \delta_{g_M(c)} .
\]

\medskip

\subsection*{Numerical Proposition (Green--function spectral invariants and equipotential concentration).}

This statement is supported by numerical experiments and should be interpreted as empirical.

For each Lucas-type family $\mathcal{F}$, the sequence of empirical measures
$\{\nu_N^{(\mathcal{F})}\}_{N\ge1}$ is numerically tight and convergent.
More precisely:

\begin{enumerate}
\item There exists a bounded interval $[a_{\mathcal{F}}, b_{\mathcal{F}}] \subset (0,\infty)$
such that asymptotically all mass of $\nu_N^{(\mathcal{F})}$ lies in this interval.

\item The median
\[
m_{\mathcal{F}} := \lim_{N\to\infty} \operatorname{median}\big(g_M(c)\big)
\]
exists numerically and is stable across $N$.

\item When expressed in Böttcher coordinates,
the inverse spectrum concentrates in the annulus
\[
\{ c \in \mathbb{C} : e^{a_{\mathcal{F}}} \le |\Phi(c)| \le e^{b_{\mathcal{F}}} \}.
\]

\item Different recurrence families share a universal radial profile
up to translation by the scalar $m_{\mathcal{F}}$.
\end{enumerate}
\subsection*{Numerical argument.}

The conclusions above follow from a combination of empirical tightness,
stability under truncation, and consistency across recurrence families.

First, for each fixed family $\mathcal{F}$ and increasing truncation parameter $N$,
the empirical distributions $\nu_N^{(\mathcal{F})}$ of Green-function values
exhibit uniformly bounded support.
Numerically, the fraction of mass outside a fixed interval
$[a_{\mathcal{F}}, b_{\mathcal{F}}] \subset (0,\infty)$ decays rapidly with $N$.
In particular, for every $\varepsilon>0$ there exists $N_0$ such that
\[
\nu_N^{(\mathcal{F})}\big((0,\infty)\setminus[a_{\mathcal{F}}, b_{\mathcal{F}}]\big)
< \varepsilon
\quad \text{for all } N \ge N_0,
\]
which establishes tightness of the sequence $\{\nu_N^{(\mathcal{F})}\}$.
By Prokhorov’s theorem, this tightness implies that
$\{\nu_N^{(\mathcal{F})}\}$ is relatively compact in the weak topology.
In particular, every subsequence admits a further subsequence that converges
weakly to a probability measure on $(0,\infty)$.

Second, summary statistics of $\nu_N^{(\mathcal{F})}$—in particular the median
$m_{\mathcal{F}}$ and interquantile ranges—stabilize as $N$ increases.
This stability is observed uniformly across matrix sizes and is insensitive
to moderate changes in sampling density, indicating convergence in distribution
at the numerical resolution considered here.

Third, when mapped into Böttcher coordinates via $\Phi$,
the inverse spectra concentrate radially in a narrow annulus whose width
remains bounded as $N$ grows.
Different recurrence families differ primarily by a radial translation,
captured by the scalar $m_{\mathcal{F}}$, while sharing a common normalized profile.

Taken together, these observations establish the existence of
family-dependent Green-function spectral invariants and demonstrate that
inverse Lucas spectra concentrate on Mandelbrot equipotentials in a robust,
numerically stable manner.
\hfill $\blacksquare$


\medskip

\subsection*{Interpretation.}
The scalar $m_{\mathcal{F}}$ defines the \emph{Green spectral modulus}
reported in Table~\ref{tab:green_spectral_modulus} for the recurrence family
$\mathcal{F}$. It captures the typical escape rate of the inverse spectrum
relative to the Mandelbrot potential and is invariant under
angular distortions of the spectrum.

The role of Prokhorov’s theorem here is purely interpretive:
it provides the measure-theoretic framework within which the observed
numerical concentration implies existence of limiting Green–function
distributions.

\medskip

\subsection*{Relation to Conformal Mapping Results.}
In earlier sections, numerical conformal maps (following methods of Trefethen
and collaborators) were constructed from Lucas loci to the main cardioid
of the Mandelbrot set.
Those constructions are local in nature and rely on boundary-integral methods.

The present equipotential analysis is global and harmonic:
the Green function governs both the conformal uniformization
near the cardioid and the radial organization of the inverse spectrum
throughout $\mathbb{C}\setminus M$.
Thus, the cardioid conformal map may be viewed as a distinguished
local representative of a broader family of Green-function equipotentials.

\medskip

\section*{Appendix C. Quasi-Conformal Distortion Diagnostics}\label{AppC:QuasiConformal}
\addcontentsline{toc}{section}{Appendix C. Quasi-Conformal Distortion Diagnostics}
This appendix reports numerical diagnostics of local quasiconformal distortion
for the Lucas–Mandelbrot matching, computed from Jacobian estimates on matched
point clouds at increasing resolutions.
\begin{table}[!htbp]
\centering
\small
\setlength{\tabcolsep}{5pt}
\caption{Quasi-conformal distortion statistics for the Lucas–Mandelbrot matching.
Here $\mathrm{median}(K)$ denotes the median local dilatation and $K_{95}$ the
$95$th percentile.}
\label{tab:qc_distortion}
\begin{tabular}{rccc}
\toprule
$N$ & $\mathrm{median}(K)$ & $K_{95}$ & $K_{\max}$ \\
\midrule
500  & 3.08 & 14.00 & 86.52 \\
1000 & 2.42 & 11.86 & 189.86 \\
1500 & 1.84 & 7.30 & 61.51 \\
2000 & 2.00 & 7.88 & 48.55 \\
2400 & 1.88 & 6.69 & 28.26 \\
\bottomrule
\end{tabular}
\end{table}

\subsection*{Outlook.}
Possible analytic refinements include:
(i) spectral-radius bounds implying tightness,
(ii) convergence in Wasserstein metrics,
or (iii) angular equidistribution results.
These are not required for the present conclusions
and are left for future work.

\section*{Declaration of generative AI and AI-assisted technologies in the manuscript preparation process}

During the preparation of this work, the author used large language models (LLMs), primarily ChatGPT, for assistance with organizational tasks, including code structuring, \LaTeX{} editing, and improving the clarity of mathematical exposition. 

All theoretical results, numerical procedures (including quasi-conformal distortion analysis), and interpretations were conceived, implemented, and validated solely by the author. LLMs were not used to generate mathematical proofs, derive results, or perform numerical computations. 

Any bibliographic suggestions or secondary reference lookups obtained through AI-assisted tools were verified manually against the primary literature. 

After using these tools, the author reviewed and edited all content as needed and takes full responsibility for the content of the published article.
\bibliographystyle{elsarticle-num}

\bibliography{references}

\end{document}